\documentclass[reqno]{amsart}
\usepackage[utf8]{inputenc}
\usepackage{amsmath}
\usepackage{amsthm}
\usepackage{amsfonts}
\usepackage{amssymb}
\usepackage{enumitem}
\usepackage{esint}

\newtheorem{theorem}{Theorem}
\newtheorem{lem}{Lemma}
\newtheorem*{cor}{Corollary}
\newtheorem{prop}{Proposition}
\newtheorem{definition}{Definition}
\newtheorem{remark}{Remark}

\newcommand{\ra}{\rightarrow}

\newcommand{\dist}{{\rm dist}}

\newcommand{\R}{{\mathbb R}}
\newcommand{\Rn}{{\R^n}}
\newcommand{\N}{{\mathbb N}}

\newcommand{\loc}{{\rm loc}}
\newcommand{\supp}{\mbox{supp}\;}
\newcommand{\diam}{\mbox{diam}}

\newcommand{\Loneloc}{{L^1_\loc}}
\newcommand{\BMO}{{\rm BMO}}
\newcommand{\bmo}{{\rm bmo}}
\newcommand{\VMO}{{\rm VMO}}
\newcommand{\vmo}{{\rm vmo}}
\newcommand{\CMO}{{\rm CMO}}

\newcommand{\cO}{{\mathcal O}}
\newcommand{\I}{{\mathcal I}}
\newcommand{\cF}{{\mathcal F}}

\newcommand{\lC}{\leq C}

\newcommand{\Omegabar}{{\overline{\Omega}}}
\newcommand{\bOmega}{{\partial\Omega}}
\newcommand{\bS}{{\partial S}}
\newcommand{\Omegatil}{{\widetilde{\Omega}}}
\newcommand{\Btil}{{\widetilde{B}}}

\newcommand{\psimu}{{\psi^\mu}}
\newcommand{\psimuS}{{\psi^\mu_S}}
\newcommand{\psimuSp}{{\psi^\mu_{S'}}}
\newcommand{\psimuSi}{{\psi^{\mu_i}_{S_i}}}
\newcommand{\psitil}{{\tilde{\psi}}}
\newcommand{\tphi}{\tilde{\phi}}
\newcommand{\etabar}{\overline{\eta}}
\newcommand{\omegabar}{\overline{\omega}}

\makeatletter
\@namedef{subjclassname@2010}{%
  \textup{2010} Mathematics Subject Classification}
\makeatother

\begin{document}

\title{On the extension of VMO functions}

\author[A. Butaev]{Almaz Butaev}
\address{(A.B.) Department of Mathematics and Statistics, Concordia University, Montreal, Quebec, H3G 1M8, Canada}
\email{almaz.butaev@concordia.ca}

\author[G. Dafni]{Galia Dafni}
\address{(G.D.) Department of Mathematics and Statistics, Concordia University, Montreal, Quebec, H3G 1M8, Canada}
\email{galia.dafni@concordia.ca}

\thanks{This work was partially funded by the Natural Sciences and Engineering Research Council (NSERC) of Canada, the Centre de recherches math\'e{}matiques (CRM)
and the Fonds de recherche du Qu\'e{}bec -- Nature et technologies (FRQNT)}

\keywords{Bounded mean oscillation, vanishing mean oscillation, extension, uniform domain}
\subjclass[2010]{42B35, 46E30}

\maketitle

\begin{abstract}We consider functions of vanishing mean oscillation on a bounded domain $\Omega$ and prove a $\VMO$ analogue
of the extension theorem of P. Jones for $\BMO(\Omega)$.  We show that if $\Omega$  satisfies the same condition imposed
by Jones (i.e.\ is a uniform domain), there is a linear extension map from $\VMO(\Omega)$ to $\VMO(\Rn)$ which is bounded in
the $\BMO$ norm.  Moreover, if such an extension map exists from $\VMO(\Omega)$ to $\BMO(\Rn)$, then the domain is uniform.
\end{abstract}

\section{Introduction}
The space of functions of bounded mean oscillation, $\BMO$, was introduced by John and Nirenberg in \cite{JN}.
Let $Q$ denote a cube in $\Rn$ with sides parallel to the coordinate axes.
The mean oscillation of an integrable function
$f$ on $Q$ is given by
$$\fint_Q  |f(x) - f_Q| dx : = \frac{1}{|Q|} \int_Q |f(x) - f_Q| dx,$$
where $|Q|$ is the volume of the cube and
$f_Q =  \fint_Q f(x) dx$ is the average of $f$ on $Q$.
For a domain $\Omega \subset \Rn$, $\BMO(\Omega)$
is the collection of those functions $f \in \Loneloc(\Omega)$ for which
\begin{equation}
\label{eqn-norm}
\|f\|_{\BMO(\Omega)}:= \sup_{Q \subset \Omega} \fint_Q |f(x) - f_Q| dx < \infty,
\end{equation}
and $\|\cdot\|_{\BMO(\Omega)}$ defines a norm which makes $\BMO(\Omega)$ a Banach space modulo constants.
It is customary to assume functions in $\BMO$ are real-valued.
This space was considered by Jones \cite{Jones}, who gave a necessary and sufficient condition on a domain $\Omega$ so that
$\BMO(\Omega)$ coincides with the restrictions to $\Omega$ of functions in $\BMO(\Rn)$, $n \geq 2$.  In what follows we will refer to this
condition, which will be defined in Section~\ref{sec-domain}, as {\em the Jones condition}.   It was shown in \cite{Gehring_Osgood} that the domains
satisfying this condition are precisely the uniform domains introduced earlier by Martio and Sarvas in \cite{Martio}. Domains of this type
are also known as $(\epsilon,\infty)$ domains \cite{Jones2} or  $1$-sided NTA domains (see e.g.\ \cite{Toro}) as they satisfy the interior conditions
of the non-tangentially accessible domains of Jerison and Kenig \cite{Jerison_Kenig}.

Functions of vanishing mean oscillation, $\VMO$, where introduced by Sarason in \cite{Sarason}.
For a bounded domain $\Omega$, we say
\begin{equation}
\label{eqn-VMO}
f \in \VMO(\Omega) \iff \lim_{t \ra 0} \omega_{\Omega}(f, t) = 0,
\end{equation}
where
\begin{equation}
\label{eqn-modulus}
\omega_{\Omega}(f, t):= \sup_{Q \subset \Omega, \ell(Q) < t} \fint_Q |f(x) - f_Q| dx.
\end{equation}
Here $\ell(Q)$ is the sidelength of the cube $Q$ and we call $\omega_{\Omega}(f,\cdot)$ the {\em modulus of mean oscillation} (see \cite{BlascoPerez}).

On $\Rn$, $\VMO(\Rn)$ can refer to the closure of the uniformly continuous functions in $\BMO$, in which case
the characterization (\ref{eqn-VMO}) remains valid (this is Sarason's definition).  A smaller space which is also sometimes referred to as $\VMO(\Rn)$,
alternatively $\VMO_0(\Rn)$ or $\CMO(\Rn)$, is
the closure in $\BMO(\Rn)$ of continuous functions with compact support as defined by Coifman and Weiss in \cite{CoifmanWeiss}.  This requires
more vanishing mean oscillation conditions - see \cite{Uchiyama, Bourdaud}.  A nonhomogeneous version of this space, denoted by $\vmo(\Rn)$,
can be defined as the closure of $C_0(\Rn)$ in the {\em local} $\BMO$ space $\bmo(\Rn)$, defined by Goldberg \cite{Goldberg}.  In addition to
 (\ref{eqn-VMO}), $f \in \vmo(\Rn)$ must also satisfy a vanishing condition on its averages over large cubes going to infinity (see \cite{Bourdaud, Dafni}):
$$\lim_{R \ra \infty}  \sup_{Q \subset \Rn\setminus B(0,R), \ell(Q) \geq 1} \fint_Q |f(x)| dx = 0.$$

In view of Jones' extension theorem, a natural question is under what conditions on $\Omega$ can we extend functions in
$\VMO(\Omega)$ to functions in $\VMO(\Rn)$. In  \cite{italians1, italians2}, extension theorems were proved for another kind
of $\BMO$ and $\VMO$ spaces on $\Omega$, where the mean oscillation is controlled over cubes/balls which have centers in $\Omega$ but
are allowed to cross the boundary. A different question was answered by Holden in \cite{Holden}: given a bounded
measurable set $E$, what are the necessary and sufficient conditions on a function $f \in \Loneloc(E)$ to be the restriction to $E$ of
a function in $\VMO(\Rn)$.  This was the analogue of an unpublished result proved by Wolff for $\BMO$.

Brezis and Nirenberg in \cite{BN2} show that any function $f$ in $\VMO(\Omega)$ satisfying
a vanishing condition on averages over small cubes approaching $\bOmega$, namely
\begin{equation}
\label{eqn-VMOz}
\lim_{\ell(Q) = \dist(Q,\bOmega) \ra 0} \fint_Q |f(x)| dx = 0,
\end{equation}
has an extension to $\VMO(B)$
which is identically zero outside $\Omega$, where $B$ is any open ball containing $\Omegabar$ (in fact such an extension will also be
in $\VMO(\Rn)$ according to any of the definitions above).  When $\Omega$ is a bounded domain with Lipschitz boundary,
they show that  condition (\ref{eqn-VMOz}) is necessary for such an extension to exist.

The space of $\VMO(\Omega)$ functions which can be extended to zero outside $\Omega$ is denoted by $\VMO_0(\Omega)$
 in Brezis and Nirenberg \cite{BN2}, and was
denoted by $\VMO_z(\Omega)$ in previous work \cite{AuscherRuss, Chang}.  It is important to clarify, first of all,
that these functions are still considered modulo constants, so that
extending by zero is the same as extending by a constant.
Moreover, there are two possible choices of norms.  When looking at $\VMO_z(\Omega)$ as contained in $\BMO_z(\Omega)$,
the subspace consisting of those functions in $\BMO(\Rn)$ which vanish on $\Rn \setminus \Omegabar$, we have that
$\VMO_z(\Omega)$ is the closure of
$C_0(\Omega)$, the continuous functions with compact support in $\Omega$, with respect to the norm $\|\cdot\|_{\BMO(\Rn)}$,
where the supremum in (\ref{eqn-norm}) is taken over all cubes in $\Rn$.  On the other hand, when considered as a subset of
$\VMO(\Omega)$ with the norm $\|\cdot\|_{\BMO(\Omega)}$, where the supremum in (\ref{eqn-norm}) is taken over all cubes $Q \subset \Omega$,
$\VMO_z(\Omega)$ is not closed.  In fact, it is proved in \cite{BN2} (attributed to Jones)
that $C_0(\Omega)$ is dense in $\VMO(\Omega)$.  When $\Omega$ satisfies the Jones condition, one can consider $\BMO(\Omega)$
as the restrictions of
$\BMO(\Rn)$ functions to
$\Omega$ (sometimes denoted $\BMO_r(\Omega)$), which makes it a quotient of $\BMO(\Rn)$ modulo functions vanishing (or rather constant)
on $\Omega$, and therefore the norm
can be taken to be the quotient norm.  The simple one-dimensional example of $\log |x|$ on $\Omega = (0,\infty)$ illustrates this difference,
since it in $\BMO(\Omega)\setminus \BMO_z(\Omega)$.

In this paper we prove the analogue of Jones' extension result for $\VMO(\Omega)$.

\begin{theorem}  Let $\Omega$ be a bounded domain in $\Rn$, $n \geq 2$.
\label{thm1}

{\rm (i)} Assume $\Omega$ satisfies the Jones condition (\ref{eqn-JonesCond}).  Then for every $f \in \VMO(\Omega)$ there exists $F \in \VMO(\Rn)$ with $F = f$ on $\Omega$.
Moreover,  the map $f \to F$ is linear and bounded:
\begin{equation}
\label{eqn-extension}
\|F\|_{\BMO(\Rn)} \leq C \|f\|_{\BMO(\Omega)},
\end{equation}
where the constant $C$ is independent of $f$.

{\rm (ii)} Suppose there exists a linear extension map which takes each $f \in \VMO(\Omega)$ to a function $F \in \BMO(\Rn)$ with $F = f$ on $\Omega$ and such that
(\ref{eqn-extension}) holds.  Then $\Omega$ satisfies the Jones condition (\ref{eqn-JonesCond}).
\end{theorem}

Since Jones' original construction of the extension gives a step function which is not in $\VMO$, the proof of part (i) involves a significant modification
of that construction which allows us to glue the pieces in a continuous way and ensure the vanishing mean oscillation condition across the boundary
$\bOmega$.  A crucial tool is the bump function constructed in \cite{BN2} in order to prove Jones' theorem on the approximation of $\VMO(\Omega)$ functions by $C^\infty_0(\Omega)$.
In Section~\ref{notation} we compute the modulus of mean oscillation for this bump function, as well as explain the Jones condition and other
properties of the domain and $\VMO$ functions.
In Section~\ref{sec-proof} we give the proof of Theorem~\ref{thm1} using some key propositions on how to glue $\VMO$ functions together.
Subsequent sections contain proofs of the propositions, which are in turn based on lemmas that are $\VMO$ analogues of the results of
Jones for $\BMO$, quantified using the modulus of mean oscillation.

\section{Notation and definitions}
\label{notation}
We will use the convention that constants may change from line to line in a series of inequalities, and that
they may depend on the dimension $n$, without pointing this out.  As usual, $A \approx B$ denotes the fact that
the ratio $A/B$ is bounded between two positive, finite constants.

\subsection{The domain}
\label{sec-domain}
We will always assume that $\Omega$ is a bounded domain, i.e.\ a bounded open and connected subset of $\Rn$.
In order to define the Jones condition, we need to fix a dyadic Whitney decomposition of $\Omega$, as in \cite{Jones}.
This means writing $\Omega$ as a countable union of dyadic cubes, which we denote by $S_j$, whose
interiors are
pairwise disjoint and whose sidelengths $\ell(S_j)$ are proportional to their distance from the complement
$\Omega^c$:
\begin{equation}
\label{eqn-distance}
\ell(S_j) \leq \dist(S_j, \Omega^c)\leq 4\sqrt{n}\; \ell(S_j).
\end{equation}
Assuming the cubes are closed, we say that $S_j$ and $S_k$ are {\em adjacent} or that they {\em touch} if
$j \neq k$ and $S_j \cap S_k \neq \emptyset$.
We also have (see \cite{Stein}, Proposition VI.1) that
\begin{equation}
\label{eqn-adjacent}
\frac 1 4  \;\ell(S_j) \leq \ell(S_k) \leq 4\ell(S_j) \quad \mbox{for adjacent }S_j, S_k.
\end{equation}
Following Jones, we denote this collection of Whitney cubes in $\Omega$ by $E$.

For $S_j,S_k \in E$, Jones defines a {\em Whitney chain} of length $m$ connecting $S_j,S_k$
to be a finite sequence of cubes in $E$, $\{Q_i\}_{i=0}^m$,
with $Q_0=S_j$, $Q_m=S_k$, and $Q_i$ touching $Q_{i+1}$ for $i=0,\dots ,m-1$. Such a chain
must always exist since $S_j, S_k$ are connected by a path which lies at a positive distance
from $\bOmega$ and therefore passes through a finite number of Whitney cubes of sidelength bounded below,
by (\ref{eqn-distance}).

We now define the two distance functions involved in the Jones condition.
\begin{definition}
\label{def-d1}
For Whitney cubes $S_j,S_k \in E$, we define  $d_1(S_j,S_k)$ to be the length of a shortest Whitney chain
connecting them.
\end{definition}

Throughout the paper, unless otherwise specified, we use $d(\cdot,\cdot)$ to denote the usual
Euclidean distance between points,
as well as the distance between points and sets, or between
two sets.  We will denote by $\log a$, $a > 0$, the logarithm to the base $2$ (instead
of the natural logarithm).

\begin{definition}
For any cubes $Q_1,Q_2$, we define the distance function $d_2(Q_1,Q_2)$ by
$$
d_2(Q_1,Q_2) := \left| \log \frac{\ell(Q_1)}{\ell(Q_2)} \right| + \log \left( 2 + \frac{d(Q_1,Q_2)}{\ell(Q_1) + \ell(Q_2)}\right).
$$
\end{definition}

Finally, we are ready to define the Jones condition.

\begin{definition}
\label{def-Jones}
We say that $\Omega$ satisfies the Jones condition if there exists a constant $\kappa > 0$ such that
\begin{equation}
\label{eqn-JonesCond}
d_1(Q_1,Q_2) \leq \kappa d_2(Q_1,Q_2) \quad \forall Q_1, Q_2 \in E.
\end{equation}
\end{definition}

A large class of domains satisfy the Jones condition.  An example of a domain that fails to satisfy
it is a slit disk: $\{r e^{i \theta}: r\in[0,1), \theta\in(0,2\pi)\}$.

In addition to the Whitney decomposition of $\Omega$, we will
use $E'$ to denote the collection of cubes in the Whitney decomposition of $\Omegabar^c$.
Since $\Omega$ is bounded, we have
\begin{equation}
\label{eqn-L}
L := \max_{S \in E} \ell(S) < \infty.
\end{equation}
For every $S'\in E'$ such that $\ell(S')\leq L$, there exists $S \in E$ with $\ell(S)\geq \ell(S')$.  We
say such a cube $S$ is a {\em matching cube} to $S'$ if it is nearest to $S'$ (in Euclidean distance).
There may be several choices for $S$.  As pointed out in \cite{Jones},  if $S$ is a matching cubes of $S'$, then
\begin{equation}
\label{matching}
\ell(S')\leq \ell(S)\leq 2 \ell(S')
\end{equation}
(otherwise $\ell(S)\geq 4 \ell(S')$ so by (\ref{eqn-adjacent}) its neighbors will have sidelength at least $\ell(S)$
and one of them will be closer to $S'$).

We denote by $\Omega'$ the union of all $S'$ in $E'$ which have matching cubes, i.e.\
$$\Omega' : = \bigcup\{S' \in E': \ell(S')\leq L\}.$$
Let
\begin{equation}
\label{eqn-Omegatil}
 \Omegatil := (\Omegabar \cup \Omega')^{\rm o},
\end{equation}
where $X^{\rm o}$ denotes the interior of the set $X$.  We will use this set for the extension in the proof of
Theorem~\ref{thm1}, so we would like to understand it well.
 By property (\ref{eqn-distance}) of the Whitney decomposition, any
 point in $\Omega'$ lies
 within distance $5\sqrt{n}L$ of $\Omegabar$, so we can take an open ball $\Btil$ with $\diam(\Btil) \approx L$ such that
$\Btil \supset \Omegabar \cup \Omega'$.
Moreover, if we take an open neighborhood of $\Omegabar$,
$$V = \{x \in \Rn: d(x, \Omegabar) < L/4\},$$
then by property  (\ref{eqn-distance}) of the Whitney decomposition,
for $S' \in E'$ with $S' \cap \overline{V} \neq \emptyset$ we must have
$\ell(S') \leq L/4$.
By property (\ref{eqn-adjacent}), any Whitney cube in $E'$ which is adjacent to such an
$S'$ must have length no larger than $L$.
Thus the layer of cubes in $E'$ covering $\partial V$ is surrounded
by another layer of Whitney cubes lying in
$\Omega' \setminus \overline{V}$, i.e.
$$\Omegabar \subset V \subset \overline{V} \subset \Omegatil \subset \Omegabar \cup \Omega' \subset \Btil.$$
The boundary $\partial \Omegatil$, lying in $\Btil\setminus\overline{V}$, is a piecewise flat surface consisting of
faces of the finitely many
cubes $S' \in E'$, $S' \subset \Omega'$, which have an adjacent cube of sidelength
greater than $L$, meaning, by \eqref{eqn-adjacent}, that
\begin{equation}
\label{eqn-OmegatilBoundary}
S' \subset \Omega' \mbox{ and } S' \cap \partial \Omegatil \neq \emptyset \implies \frac L 2 \leq \ell(S') \leq L.
\end{equation}

\subsection{The modulus of mean oscillation}
From  Theorem A1.1 in \cite{BN2} (due to Peter Jones) and the equivalence of the $\ell^2$ and $\ell^\infty$ metrics (i.e.
balls and cubes) in $\Rn$, we can replace the
definition (\ref{eqn-modulus}) of the modulus of mean oscillation of a function $f \in \Loneloc(\Omega)$ by the following equivalent form:
\begin{equation}
\label{eqn-newmodulus}
\omega_\Omega(f,t) :=
\sup_{Q \in \I(\Omega), \ell(Q) < t} \frac{1}{|Q|^2} \int_{Q} \int_{Q} |f(x)-f(y)| dx dy,
\end{equation}
where we use $\I(\Omega)$ to denote the collection of {\em interior} cubes in $\Omega$, namely those cubes
$$\I(\Omega):= \{Q \subset \Omega \; |\: 0 < \diam(Q) \leq d(Q,\bOmega)\}$$
(what is denoted in  \cite{BN2} by ${\mathcal C}_{1/2} = {\mathcal C}$).
Here we can interpret the diameter and the distance in either the $\ell^2$ or the $\ell^\infty$ metric.
Note that $\I(\Omega)$ consists of exactly those cubes for which $2Q \subset \Omega$, where $2Q$ denotes the cube with the same
center as $Q$ and twice the sidelength.

It is noted at the beginning of Appendix 1 in \cite{BN2} that these results are valid in any bounded open
set $\Omega$.
In \cite{Buckley}, the equivalence of the definition of $\BMO(\cO)$ for an open set $\cO$  by
(\ref{eqn-norm}), or by replacing $Q \subset \cO$ in the supremum with $2Q \subset \cO$,
 is attributed to \cite{ReimannRychener}.

\begin{remark}
\label{rmk-opensets}
In what follows we will also want to refer to $f \in \VMO(\cO)$ for an open set that is not necessarily connected.  We will take
this to mean that $\omega_\cO(f,t)$, defined as in (\ref{eqn-newmodulus}), is bounded and goes to zero as $t \ra 0$.
Equivalently, $f$ is in $\VMO$ on every connected component of $\cO$.
Note that the zero elements in this case are no longer constants but functions which are constant on each connected
component of $\cO$.
\end{remark}

\begin{remark}
\label{rmk-interior}
Since $\I(\Omega) = \I(\Omegabar)$, in (\ref{eqn-newmodulus}) we can replace $\Omega$ by $\Omegabar$ without changing anything.  Thus,
with an abuse of notation, we will often write below $\omega_S(f,t)$ where $S$ is a cube, which we have assumed
previously to be a closed set, when we really mean $S^0$, the interior of $S$.  Similarly, we will also write
$\VMO(S)$ for $\VMO(S^0)$.
\end{remark}

We also note that $\omega_\Omega(f,t)$ is invariant under dilations.  That is,  for  $\lambda > 0$, if we consider the image of
$\Omega$ under the dilation $x \ra \lambda x$, denoted by
$\lambda\Omega$, and the function $f_\lambda$ defined on $\lambda\Omega$ by $f_\lambda(x) = f(\lambda^{-1} x)$, we have
\begin{equation}
\label{eqn-dilation}
\omega_{\lambda\Omega}(f_\lambda,\lambda t) = \omega_\Omega(f,t).
\end{equation}
Similarly we have invariance under translations $x \ra x + c$.

Finally, we will need the fact that vertical truncations reduce the modulus of mean oscillation.  By this we mean that if $f \in \Loneloc(\Omega)$
and $m, M \in \R$, then
\begin{equation}
\label{eqn-truncation}
\omega_\Omega(\max\{f,m\} ,t) \leq \omega_\Omega(f,t) \; \mbox{and }\omega_\Omega(\min\{f,M\} ,t) \leq \omega_\Omega(f,t),
\end{equation}
since the truncation reduces the differences in the integrand on the right-hand-side of (\ref{eqn-newmodulus}).

The modulus of mean oscillation of a $\VMO$ function is an example of a modulus of continuity.
\begin{definition}
\label{def-moc}
A bounded, nondecreasing function $\eta: [0,\infty) \ra [0,\infty)$ which is continuous at $0$, with $\eta(0)=0$,
will be called a modulus of continuity.
\end{definition}
\begin{remark}
\label{rmk-moc}
While natural moduli of continuity associated with measuring the smoothness of functions, such as the modulus of mean oscillation, are nondecreasing,
we can control any $\eta$ which is not necessarily nondecreasing by a nondecreasing modulus of continuity defined
as $\tilde{\eta}(t) = \displaystyle{\sup_{s \leq t}\eta(s)}$, without changing the $L^\infty$ norm.
\end{remark}

In the following we will need the notion of the {\em least concave majorant} of a nonnegative function $\eta$,
 which is  a concave function $\etabar \geq \eta$ such that for every
concave $\theta \geq \eta$, $\etabar \leq \theta$.  As in  \cite{DevoreLorentz} (see p.\ 43), we can define such a
function by taking
$$\etabar(t) = \inf_{\mbox{\scriptsize line}\; l \geq \eta} l(t).$$
If $\eta :[0,\infty) \ra [0,\infty)$ is nondecreasing, any line  dominating $\eta$ must have nonnegative slope,
and therefore the least concave majorant $\etabar$ must be nondecreasing.  Moreover, if $\eta$ is bounded by $M$, taking $l$ to be the constant
function $M$, we have that $\etabar$ is bounded
by $M$.  Thus $\|\etabar\|_\infty \leq \|\eta\|_\infty$.  Finally, if $\phi$ is continuous at $0$, with $\eta(0)=0$, so is $\etabar$
(see p.\ 43 in \cite{DevoreLorentz}).  Thus we have the following (note that this is not Lemma 6.1 on p.\ 43 in \cite{DevoreLorentz}
since there a modulus of continuity is assumed to be subadditive):

\begin{lem}
\label{lem-lcm}
The least concave majorant of a modulus of continuity is also a modulus of continuity and has the same supremum.
\end{lem}

\subsection{Properties of $\VMO$ functions}

In order to work with $\VMO$ instead of $\BMO$, we present ``quantified" versions of some lemmas in \cite{Jones}, where
instead of using the $\BMO$ norm we use the modulus of mean oscillation.

\begin{lem}
\label{lem_d1}
Let $\Omega$ be a domain, $\phi\in \VMO(\Omega)$  and $Q_1,Q_2$ be Whitney cubes of $\Omega$. Then for some $C>0$
$$
|\phi_{Q_1} - \phi_{Q_2}| \lC  d_1(Q_1,Q_2) \cdot \omega_{\Omega}(\phi,\ell(Q)),
$$
where $Q$ is a largest cube in a shortest Whitney chain connecting $Q_1$ and $Q_2$ and
$$
\ell(Q) \leq 4^{d_1(Q_1,Q_2)} \ell(Q_1)
$$
\end{lem}

\begin{lem}
\label{lem_d2_proper}
There exists $C>0$ such that for any $\phi\in \VMO(\Omega)$ and $Q_2\subset Q_1 \subset \Omega$,
$$
 |\phi_{Q_1} - \phi_{Q_2}| \leq C \log\left(2+\frac{\ell(Q_1)}{\ell(Q_2)}\right)
\omega_\Omega(\phi,\ell(Q_1)).
$$
\end{lem}

The proofs of these results are almost the same as of Lemma 2.2 and Lemma 2.1 in \cite{Jones}, respectively, so we omit them.
We will also need a version of Lemma 2.3 in \cite{Jones} which is quantified using the notion of moduli of continuity.

\begin{lem}
\label{lem_main_one}
Let $\cO$ be an open set in $\Rn$ and $\phi$ be a locally integrable function on $\cO$. If there exist moduli of continuity $\delta_1,\delta_2:[0,\infty)\to [0,\infty)$ such that for all $s>0$
\begin{equation}
\label{temp_lem_key_1}
\sup_{\substack{\text{dyadic } Q\subset \cO \\ \ell(Q)\leq s}}\fint_Q |\phi(x) - \phi_Q| dx \leq \delta_1(s)
\end{equation}
and
\begin{equation}
\label{temp_lem_key_2}
\sup_{\substack{\text{adjacent dyadic } Q_1,Q_2\subset \cO \\ \ell(Q_1)=\ell(Q_2)\leq s}} |\phi_{Q_1} - \phi_{Q_2} | \leq \delta_2(s),
\end{equation}
then $\phi\in \VMO(\cO)$ and
$\omega_{\cO}(\phi)(s) \leq C (\delta_1(s)+\delta_2(s))$, for some constant $C>0$.
\end{lem}

\begin{proof}
In the proof of Lemma 2.3 in \cite{Jones}, we consider a cube $Q \subset \cO$ with $\ell(Q) \leq s$.
Noting that all the dyadic cubes involved in the proof will also have sidelength
less than $s$ and are contained in $Q$ (since they consist of the Whitney decomposition of $Q$), and replacing the
constants $c_i$ by $\delta_i(s)$, $i = 1,2$, we get that $\fint_Q |\phi(x) - \phi_Q| dx \leq C (\delta_1(s)+\delta_2(s))$.
\end{proof}

A corollary of this result is the following  lemma, which is used in the proof of Theorem~\ref{thm1} but is also of independent interest.
It allows us to ``glue together" $\VMO$ functions defined on cubes.

\begin{lem}
\label{lem_gluing_dyadic}
Let $\{S_i\}$ be a countable collection of dyadic cubes with disjoint interiors, and  let $\phi_{i}$ be a collection of  functions with $\phi_i \in \VMO(S_i)$.
If there exist moduli of continuity $\eta_1,\eta_2$ such that for all $i$
\begin{equation}
\label{cond_gluing_1}
\omega_{S_i}(\phi_i,t) \leq \eta_1(t),\ \forall t\geq 0
\end{equation}
and for any two adjacent cubes $S_i,S_j$
\begin{equation}
\label{cond_gluing_2}
\sup_{\substack{\text{adjacent dyadic } Q_1\subset S_i,Q_2\subset S_j \\ \ell(Q_1)=\ell(Q_2)\leq t}} |(\phi_i)_{Q_1} - (\phi_j)_{Q_2} |
\leq \eta_2(t),
\end{equation}
 then the function $\phi$ defined on $\bigcup S_i$ by
$$
\phi(x)  = \phi_i(x),   \quad x\in S_i
$$
is a element of $\VMO(\cO)$, where $\cO$ is the interior of $\bigcup S_i$, with
$$
\omega_{\cO}(\phi,t) \leq C\max(\eta_1(t), \eta_2(t))).
$$
\end{lem}

\begin{proof} We want to apply Lemma~\ref{lem_main_one}.
First note that every dyadic cube  $Q\subset \cO$ must be contained
in some $S_i$, and therefore
$$ \fint_Q |\phi(x) - \phi_Q| dx =  \fint_Q |\phi_i(x) - (\phi_i)_Q| dx\leq \omega_{S_i}(\phi_i,s) \leq \eta_1(s)$$
whenever $\ell(Q) \leq s$.  Thus (\ref{temp_lem_key_1}) holds with $\delta_1 = \eta_1$.

For (\ref{temp_lem_key_2}),  suppose $Q_1$ and $Q_2$ are adjacent dyadic cubes in $\cO$ with
$\ell(Q_1)=\ell(Q_2)\leq s$.  Then either they are both contained in a single $S_i$ or they are
contained in adjacent cubes $S_i$ and $S_j$.  In the latter case we have, by hypothesis,
$$|\phi_{Q_1} - \phi_{Q_2} | = |(\phi_i)_{Q_1} - (\phi_j)_{Q_2} | \leq \eta_2(s).$$

In the first case, by choosing adjacent subcubes $Q'_1 \subset Q_1$ and $Q'_2 \subset Q_2$ with half the side
length, and a cube $Q'$ with $\ell(Q) = \ell(Q_1)$ such that $Q'_1 \cup Q'_2 \subset Q' \subset Q$, we have, as in the proof of Lemma 2.2 in \cite{Jones}, that
\begin{eqnarray*}
\lefteqn{|\phi_{Q_1} - \phi_{Q_2} |  =   |(\phi_i)_{Q_1} - (\phi_i)_{Q_2} | }\\
&  \leq & |(\phi_i)_{Q_1} - (\phi_i)_{Q'_1} | + |(\phi_i)_{Q'_1} - (\phi_i)_{Q'} | + |(\phi_i)_{Q'} - (\phi_i)_{Q'_2} | + |(\phi_i)_{Q'_2} - (\phi_i)_{Q_2} |\\
& \leq & C\omega_{S_i}(\phi_i,s)  \leq \eta_1(s)
\end{eqnarray*}
by condition (\ref{cond_gluing_1}).  Thus (\ref{temp_lem_key_2}) holds with $\delta_2= \max(\eta_1,\eta_2)$.
\end{proof}

\subsection{Bump functions}
\label{sec-bump}
As mentioned in the introduction, in \cite{BN2} Brezis and Nirenberg quote a result
of Jones stating that any function in $\VMO(\Omega)$ can be approximated by functions in $C^\infty_0(\Omega)$.
To prove this, they construct a bump function (see Appendix 1 in \cite{BN2}) which we will adapt to our purposes.

We fix $S_0$ to be the cube in $\Rn$ centered at the origin with $\ell(S)=4$, i.e.\  $S_0 = [-2,2]^n$.
For $\mu \in \N$, we define the Brezis-Nirenberg bump function $\psimu$ on $S_0$ as follows:
$$
\psimu(x) = \left\{ \begin{array}{cc}
                          \left(1-\frac1\mu \log(\frac{1}{d(x,\partial S_0)}) \right)_+ & \text{ if $d(x,\partial S_0)<1$,} \\
                          1 & \text{ otherwise,}
                        \end{array}
    \right.
\quad \forall\; x \in S_0.
$$

Given an arbitrary cube $S$ centered at $c_S$ we put $\psimuS(x) = \psimu\left(\frac{4(x-c_S)}{\ell(S)} \right)$ or equivalently
$$
\psimuS(x) = \left\{ \begin{array}{cc}
                          \left(1-\frac1\mu \log(\frac{\ell(S)}{4 d(x,\bS)}) \right)_+ & \text{ if $d(x,\bS)<\ell(S)/4$,} \\
                          1 & \text{ otherwise,}
                        \end{array}
    \right.
\quad \forall\; x \in S.
$$
Notice that in a cube $S$ the function $d(x,\bS)$ is the same in both the $\ell^2$ and $\ell^\infty$ metrics.
Then $\psimuS$ is a continuous function which is identically equal to $1$ on a subcube $J(S)$ concentric with
$S$ and of sidelength $\ell(S)/4$, and is supported in a subcube  $K(S)$ at distance $2^{-\mu-2} \ell(S)$ from the boundary
(recall that we are using
the logarithm to the base $2$), so
$\psimuS \in \VMO(S)$ and as shown in \cite{BN2}, $\|\psimuS\|_{\BMO}= {\mathcal O}(\mu^{-1})$. We want to
refine this to obtain a better estimate for the modulus of mean oscillation of $\psimuS$, since we will need it in what follows.

\begin{lem}
\label{lem_bump_modulus}
There exists a modulus of continuity $\theta$ such that for every cube
$S$,
\begin{equation}
\label{eqn-theta}
\omega_S(\psimuS,t) \leq  \frac{1}{\mu} \theta \left(2^\mu \frac{t}{\ell(S)}\right).
\end{equation}
\end{lem}
\begin{remark}
As will follow from the proof, the result holds for any set $S$, not necessary a cube, with the appropriate adjustments
to the definition of $\psimuS$.
\end{remark}

\begin{proof}
By the translation and dilation invariance of the modulus of mean oscillation (see (\ref{eqn-dilation}) above), it is enough to
consider the case $S = S_0$.  We therefore
denote $\psimuS$ by $\psimu$ and write, by (\ref{eqn-newmodulus})
$$
\omega_S(\psimu,t)
=\sup_{Q \in \I(S),  \ell(Q) < t} \frac{1}{|Q|^2} \int_{Q} \int_{Q} |\psimu(x)-\psimu(y)| dx dy.
$$
As noted following the definition, the function $\psimu$ vanishes outside the cube
 $K=\{x\in S: d(x,\bS)\geq2^{-\mu}\}$.
Thus
$$
\omega_S(\psimu,t) =\sup_{Q \in \I(S),  \ell(Q) < t, Q \cap K \neq \emptyset} \frac{1}{|Q|^2} \int_{Q} \int_{Q} |\psimu(x)-\psimu(y)| dx dy.
$$
Note that the function $\psimu$ is a truncation of
$$\psitil(x) =\frac1\mu (\mu+\log d(x,\bS))$$
by $1$ from above and $0$ from below.
Therefore
\begin{eqnarray*}
\omega_S(\psimu,t)
& \leq & \sup_{Q \in \I(S),  \ell(Q) < t, Q \cap K \neq \emptyset} \frac{1}{|Q|^2} \int_{Q} \int_{Q} |\psitil(x)-\psitil(y)| dx dy\\
& = & \sup_{Q \in \I(S),  \ell(Q) < t, Q \cap K \neq \emptyset}  \frac{1}{\mu |Q|^2} \int_{Q} \int_{Q} |\log d(x,\bS)-\log d(y,\bS)| dx dy
\\
& \leq & \frac{C}{\mu}  \sup_{Q \in \I(S),  \ell(Q) < t, Q \cap K \neq \emptyset} \frac{\ell(Q)}{d(Q,\bS)} \leq  \frac{C}{\mu}  \min(\frac 12, 2^{\mu+1} t).
\end{eqnarray*}
Here we have used the fact that if $Q \in \I(S)$ then $\ell (Q) \leq d(Q,\bS)/2$ and if furthermore $Q \cap K \neq \emptyset$ then
$$2^{-\mu} = d(K,\bS) \leq \diam(Q) + d(Q,\bS) \leq 2 d(Q,\bS).$$
Replacing $t$ by $4t/\ell(S)$ for a general cube $S$, we get (\ref{eqn-theta}) with $\theta$ a piecewise linear function.
\end{proof}
	
\section{Proof of Theorem~\ref{thm1}}
\label{sec-proof}
The proof of the theorem is based on three propositions which allow us to glue together the bump functions
introduced in the previous section in order to form VMO functions.

We formulate the propositions  in the following section, then prove the theorem and finally prove the propositions.

\subsection{Key propositions}

\begin{prop}
\label{main_lemma_1}
Given countable collections of numbers $\lambda_i\in \R$, $\mu_i\geq 1$,
and dyadic cubes $S_i$ with pairwise disjoint interior
 satisfying (\ref{eqn-adjacent}), set
\begin{equation}
\label{def_Psi_first_time}
\Psi(x)  = \lambda_i \psi^{\mu_i}_{S_i}(x), \quad x\in S_i,
\end{equation}
where $\psi^{\mu_i}_{S_i}$ are bump functions defined in Section~\ref{sec-bump}. Suppose that the following conditions are satisfied:
\begin{itemize}
\item
There exist $C_1>0$ and $C_2\in \mathbb{R}$ such that for all $i$
\begin{equation}
\label{cond_0}
 \mu_i \geq C_1 \log \frac{L}{\ell(S_i)} - C_2,
\end{equation}
where $L = \sup \ell(S_i)$.
\item There exists a $C>0$ such that for all adjacent cubes $S_i$, $S_j$
\begin{equation}
\label{cond_1}
|\mu_i-\mu_j|\leq C
\end{equation}
\item There exists a modulus of continuity $\delta_1$ such that for all $i$
\begin{equation}
\label{cond_2}
\frac{|\lambda_i|}{\mu_i} \leq \delta_1 \left(\frac{1}{\mu_i}\right)
\end{equation}
\item There exists a modulus of continuity $\delta_2$ such that for all adjacent cubes $S_i$, $S_j$
\begin{equation}
\label{cond_3}
|\lambda_i - \lambda_j| \leq \delta_2\left(\frac{1}{\mu_i}+\frac{1}{\mu_j}\right),
\end{equation}
\end{itemize}
Then $\Psi\in \VMO(\cO)$, where $\cO  = (\bigcup S_i)^{\rm o}$,  and
$$
\|\Psi\|_{\BMO} \leq C (\|\delta_1\|_{\infty}+ \|\delta_2\|_{\infty}).
$$
\end{prop}

\begin{cor}
If $\mu_i=1+\log\frac{L}{\ell(S_i)}$ and for some moduli of continuity $\delta_1,\delta_2$ we have
$$
\frac{|\lambda_i|}{\mu_i} \leq \delta_1 \left(\ell(S_i)\right) , \forall i
$$
and
$$
|\lambda_i - \lambda_j| \leq \delta_2(\ell(S_i)+\ell(S_j)),
$$
for all adjacent cubes $S_i$, $S_j$, then $\Psi\in \VMO((\bigcup S_i)^{\rm o})$ and
$$
 \|\Psi\|_\BMO \leq C (\|\delta_1\|_\infty +\|\delta_2\|_\infty).
$$
 \end{cor}

\begin{prop}
\label{main_lemma_2}
 Let $\Omega$ be a bounded domain satisfying the Jones condition (\ref{eqn-JonesCond}), $S,S_L\in E$, the collection of Whitney cubes of $\Omega$,
  and $S_L$ is a cube of maximum sidelength $\ell(S_L)=L$ (see \eqref{eqn-L}).
Let $\phi\in \VMO(\Omega)$ and $\omegabar$ be the least concave majorant of $\omega_{\Omega}(\phi,\cdot)$.
Then there exists $C_\Omega>0$ such that
$$
|\phi_S - \phi_{S_L}| \leq C_\Omega \left( 1+ \log\frac{L}{\ell(S)}\right)  \cdot \omegabar\left(\frac{6L }{1+\log\frac{L }{\ell(S)}} \right).
$$
\end{prop}

\begin{prop}
\label{main_lemma_3}
Let $\Omega$ be a domain satisfying the Jones condition (\ref{eqn-JonesCond}),
$\phi_1\in \VMO(\Omega)$ and $\phi_2\in \VMO(\Omega')$.
If there exists a modulus of continuity $\eta$ such that for each $S'\in E'$ with $S' \subset \Omega'$,
and for some $S \in E$ which is a matching cube of $S'$,
\begin{equation}
\label{eqn-lemma3hyp1}
|(\phi_1)_{S} - (\phi_2)_{S'}| \leq \eta(\ell(S')),
\end{equation}
then
$$
\Phi(x) = \left\{
\begin{array}{cc}
        \phi_1(x), &  x\in \Omega\\
        \phi_2(x), & x\in \Omega'
\end{array}
\right.
$$
is an element of $\VMO(\Omegatil)$, where $\Omegatil$ is defined as in (\ref{eqn-Omegatil}), and for some constant $C>0$
$$
\|\omega_{\Omegatil}(\Phi, \cdot)\|_\infty \leq C (\|\omega_{\Omega}(\phi_1, \cdot)\|_\infty+\|\omega_{\Omega'}(\phi_2,\cdot)\|_\infty+ \|\eta\|_\infty).
$$
\end{prop}

\subsection{Proof of Theorem~\ref{thm1} part (i)}

Given $\phi\in \VMO(\Omega)$, where $\Omega$ is a bounded domain satisfying Jones' condition (\ref{eqn-JonesCond}),
we need to extend it to a function in $\VMO(\Rn)$, and show that the extension operator is bounded linear operator
in the $\BMO$ norm.  This will be done in two steps.

\begin{itemize}

\item \underline{Main extension} Extend $\phi$, in a linear fashion,
to a function a function $\tphi$ which is in $\VMO$ with compact support in $\Omegatil$ and
show
\begin{equation}
\label{eqn-bmotphi}
\|\tphi\|_{\BMO(\Omegatil)} \lesssim \|\phi\|_{\BMO(\Omega)},
\end{equation}
with a constant depending only on $\Omega$.

\item \underline{Further extension}  Extend $\tphi$ by zero outside $\Omegatil$ and show it
is in $\VMO(\Rn)$ with
$$\|\tphi\|_{\BMO(\Rn)} \leq C \|\tphi\|_{\BMO(\Omegatil)},$$
with a constant depending only on $\Omega$.
 \end{itemize}

 We now describe the construction of the main extension.
By Corollary 2.9 in \cite{Jones}, $\bOmega$ has measure zero.  Therefore in order to define $\tphi$ almost everywhere on $\Omegatil$, we only need to define it on $\Omega$ and $\Omega'$.
 By analogy with the definition of the extension in \cite{Jones}, we set
$$
 \tphi(x) = \left\{  \begin{array}{cc}
        \phi(x), &  x\in \Omega,\\
        \lambda_i \psi^{\mu_i}_{S'_i}(x),  & x\in S'_i,
    \end{array}
    \right.
$$
where for each Whitney cube $S'_i$ contained in $\Omega'$ we fix a matching cube $S_i \in E$, $\psi^{\mu_i}_{S'_i}$ is the bump function defined
in Section~\ref{sec-bump}, and
$$\mu_i = 1+\log\frac{L}{\ell(S'_i)},  \quad \lambda_i = \phi_{S_i}.$$

First we check that this extension is linear in $\phi$.  That follows from the fact that the only dependence on $\phi$ is in $\lambda_i$.
We must also address the fact that the function $\phi \in \VMO(\Omega)$ is only defined up to constants.  In Jones' construction in \cite{Jones}, the extension is also modulo constants.
However, when we use the bump functions, we consider that they have compact support, and thus we must fix the function $\phi$ by assuming that
\begin{equation}
\label{eqn-SL}
\phi_{S_L} = 0,
\end{equation}
where we have picked one of the cubes in $E$ of sidelength $L$ and called it $S_L$.  If we add a constant to $\phi$, we will then need to add
that same
constant in the definition of $\tphi$ on $\Omega'$.

We claim that $\tphi$ has compact support in $\Omegatil$ with
\begin{equation}
\label{eqn-supptphi}
d(\supp(\tphi), \partial \Omegatil) \approx L.
\end{equation}
By the properties of the bump functions, we have that
$$\supp(\tphi) \subset \bigcup_i K(S'_i) \cup \Omegabar  \subset \Omegatil,$$
where the subcube $K(S'_i)$ of $S'_i$ contains the support of $\psi^{\mu_i}_{S'_i}$
and lies at a distance of $2^{-\mu-2}\ell(S'_i)$ from $\partial S'_i$.

By property \eqref{eqn-OmegatilBoundary} of the boundary of $\Omegatil$, we have
\begin{eqnarray*}
d(\supp(\tphi), \partial \Omegatil) & = & \min_{S' \in E', \frac L{2} \leq \ell(S') \leq L} d(K(S'),\partial S') \\
 & = &  \min_{S' \in E', \frac L{2} \leq \ell(S') \leq L} 2^{-\mu-2}\ell(S')\\
& \approx & \min_{S' \in E', \frac L{2} \leq \ell(S') \leq L} \frac{\ell(S')}{L}\ell(S')
\approx L.
\end{eqnarray*}

Let us now show how to go to the further extension from the main extension.
Assuming $\tphi \in \VMO(\Omegatil)$, since it has compact support, by Remark 6 preceding
Theorem 2 in \cite{BN2} it will have a zero extension to $\VMO(\Btil)$,
and hence to $\VMO(\Rn)$, as noted in the
introduction.  This is not dependent on any smoothness of the boundary of $\Omegatil$, and in fact is valid for any bounded open
set, not necessarily a domain.
Moreover,  by Remarks 1 and 6 in \cite{BN2}, the zero extension of $\tphi$ will be in
$\BMO(\Btil)$, hence in $\BMO(\Rn)$, with
$$\|\tphi\|_{\BMO(\Rn)} \leq C \|\tphi\|_{\BMO(\Omegatil)},$$
where $C$ depends on how far the support of $\tphi$ lies from
$\partial \Omegatil$, hence by (\ref{eqn-supptphi}), only on the constant $L$.
Finally, we note that this is a linear extension, and as pointed out above, if we modify $\tphi$ in $\Omegatil$ by adding a constant,
the extension will not be a zero extension but rather a constant extension.

It now remains to show how the main extension follows from the propositions in the previous section.

We first apply the Corollary of Proposition~\ref{main_lemma_1} to show that $\tphi\in \VMO(\Omega')$,
where again with an abuse of notation we use $\Omega'$ to mean its interior
(see Remarks~\ref{rmk-opensets},\ref{rmk-interior}).  Let us verify the hypotheses of the Corollary for the cubes $S'_i$ making up $\Omega'$.
For each $i$, let $S_i$ be the cube which we chose as a matching
cube to $S'_i$.
Note that by (\ref{matching}) we can exchange $\ell(S'_i)$ and $\ell(S_i)$ up to a factor of $2$.  We first
apply Proposition~\ref{main_lemma_2} to $S = S_i$ and the cube $S_L$ for which we assumed vanishing mean in (\ref{eqn-SL})
to get
$$
\frac{|\lambda_i|}{\mu_i} = \frac{|\phi_{S_i}|}{1+\log\frac{L}{\ell(S'_i)}} \approx
\frac{|\phi_{S_i}|}{1+\log\frac{L}{\ell(S_i)}} \leq \delta_1(\ell(S_i)):= C_\Omega \omegabar\left(\frac{6 L }{(1+\log\frac{L }{\ell(S'_i)})} \right)
$$
with $\omegabar$ the least concave majorant of $\omega_\Omega(\phi,\cdot)$.  Using Lemma~\ref{lem-lcm}, we have that $\delta_1$
is a modulus of continuity, with $\|\delta_1\|_\infty \leq C_\Omega \|\phi\|_{\BMO}$.
Furthermore, for adjacent $S'_i,S'_j$, by Lemma \ref{lem_d1}
$$
|\lambda_i-\lambda_j|=|\phi_{S_i}-\phi_{S_j}| \leq C_\Omega \cdot \omega_\Omega(\phi,\ell(S'_i)+\ell(S'_j)).
$$
Thus we can apply the Corollary to conclude that $\tphi$ is in $\VMO(\Omega')$ with
\begin{equation}
\label{eqn-tphiBMO}
\|\tphi\|_{\BMO(\Omega')} \leq C (\|\delta_1\|_\infty + \|\omega_\Omega(\phi, \cdot)\|_\infty \leq C\|\phi\|_{\BMO(\Omega)}.
\end{equation}

The theorem will be proved if we can apply Proposition~\ref{main_lemma_3} with $\phi_1 = \phi$ on $\Omega$ and
$\phi_2 = \tphi$ on $\Omega'$.  Note that for any $S'$ in $\Omega'$, taking the matching $S\in E$ to be the one chosen in the
construction of $\tphi$, we have
$$
|\phi_S - \tphi_{S'}| = |\phi_S| |1-(\psimuSp)_{S'} |, \quad \mu = 1+\log(L/\ell(S')).
$$
Recalling that $\psimuSp$ is identically $1$ on a subcube $J(S') \subset S'$  which is concentric with $S'$ and of sidelength $\ell(S')/4$, we can write the last equality as
$$
|\phi_S - \tphi_{S'}| = |\phi_S|  |(\psimuSp)_{J(S')}-(\psimuSp)_{S'} |.
$$
From Lemmas \ref{lem_d1} and \ref{lem_bump_modulus}, we get
 $$
|\phi_S - \tphi_{S'}| \leq  \frac{C|\phi_S|}{1+\log\frac{L}{\ell(S')}}.
$$
Combining this with Proposition~\ref{main_lemma_2}, applied to the cube $S$ and the cube $S_L$ with vanishing mean as in (\ref{eqn-SL}), we obtain
$$
 |\phi_S - \tphi_{S'}| \leq C \eta(\ell(S'))
$$
for $\eta(t) = \omegabar(6L(1 + \log (L/t))^{-1})$, where $\omegabar$ is the least concave majorant of $\omega_\Omega(\phi,\cdot)$.  By
Lemma~\ref{lem-lcm} we have that $\eta$ is a modulus of continuity with $\|\eta\|_\infty \leq \|\omega_\Omega(\phi,\cdot)\|_\infty = \|\phi\|_{\BMO(\Omega)}$,
 so we can apply Proposition~\ref{main_lemma_3} and (\ref{eqn-tphiBMO}) to conclude that $\tphi \in \VMO(\Omegatil)$ with
$$\|\tphi\|_{\BMO(\Omegatil)} \leq C (\|\phi\|_{\BMO(\Omega)}+\|\tphi\|_{\BMO(\Omega')}  + \|\eta\|_\infty)
\leq C\|\phi\|_{\BMO(\Omega)}.$$
This shows the boundedness of the main extension and completes the proof of part (i) of Theorem~\ref{thm1}.

\subsection{Proof of Theorem~\ref{thm1} part (ii)}
As previously stated, $\Omega$ is a bounded domain.  In this section, we assume that there is a bounded linear extension
$$\Lambda: \VMO(\Omega) \ra \BMO(\Rn)$$
and we want to show
this implies that $\Omega$ satisfies Jones' Condition  (\ref{eqn-JonesCond}).

We first fix two Whitney cubes $S_1,S_2\in E$, and define a function $\phi_{S_1,S_2}$ on $\Omega$ as follows.  Recalling the bump
functions defined in Section~\ref{sec-bump}, on each $S_i\in E$ set
$$
\phi_{S_1,S_2}(x) = \lambda_i \psimuSi(x),\quad x \in S_i,
$$
where
$$ \lambda_i = \delta \left(\frac{1+d_1(S_1,S_2)}{1+d_1(S_1,S_i)}\right) d_1(S_1,S_i),\quad  \mu_i=1+d_1(S_1,S_i).$$
As in Definition~\ref{def-d1}, $d_1(S_1,S_2)$ is the length of the shortest Whitney chain between $S_1$ and $S_2$,
and we fix $\delta$ to be some smooth nonnegative nondecreasing function with $\delta(t)=0$ for $t\leq 1/2$ and $\delta(t)=1$ for $t\geq 1$.
We want to apply the extension to $\phi_{S_1,S_2}$, so we need to show it is in $\VMO$ on the domain.

\begin{lem}
\label{lem_for_th1}
$\phi_{S_1,S_2}\in \VMO(\Omega)$ and
$$
\|\phi_{S_1,S_2}\|_\BMO \leq C,
$$
uniformly in $S_1,S_2$, for some $C>0$.
\end{lem}

\begin{proof}
By Proposition~\ref{main_lemma_1}, it suffices to check that conditions (\ref{cond_0})-(\ref{cond_3}) hold.
Note that by the nature of Whitney chains and property (\ref{eqn-adjacent}) of Whitney cubes, if $d_1(S_1,S_i)\leq \alpha$, then
$$4^{-\alpha} \leq \frac{\ell(S_1)}{\ell(S_i)}\leq 4^\alpha.
$$
Hence
$$
d_1(S_1,S_i) \geq \frac{1}{2} \left|\log \frac{\ell(S_1)}{\ell(S_i)} \right| \geq \frac{1}{2}\log \frac{L}{\ell(S_i)} - \frac{1}{2}\log \frac{L}{\ell(S_1)}
$$
 and condition (\ref{cond_0}) holds. Condition (\ref{cond_1}) holds trivially with $C=1$ and (\ref{cond_2}) is true by construction with $\|\delta_1\|_{\infty} = \|\delta\|_\infty$ as
$$
 \frac{|\lambda_i|}{\mu_i} \leq \delta \left(\frac{1+d_1(S_1,S_2)}{1+d_1(S_1,S_i)}\right).
$$
Finally, if $S_i$ and $S_j$ are adjacent cubes, set $A = 1+d_1(S_1,S_2)$, $B_i = d_1(S_1,S_i)$ and $B_j = d_1(S_1,S_j)$.  Then $B_i - B_j = \pm 1$ so assume, without loss of generality,
that $B_i - B_j = 1$ and write
\begin{eqnarray}
\nonumber
|\lambda_i-\lambda_j| & = &  \left|\delta \left(\frac{A}{1+B_i}\right)B_i  -  \delta \left(\frac{A}{1+B_j}\right)B_j\right|\\
& \leq &  \delta \left(\frac{A}{1+B_i}\right)\ + \left|\delta \left(\frac{A}{1+B_i}\right) -  \delta \left(\frac{A}{1+B_j}\right)\right|B_j.
\label{temp_00_01}
\end{eqnarray}
Applying the mean value theorem, we get that
\begin{eqnarray*}
\left| \delta \left(\frac{A}{1+B_i}\right) - \delta \left(\frac{A}{1+B_j}\right) \right|
& = & \left| \delta \left(\frac{A}{1+B_i}\right) - \delta \left(\frac{A}{1+B_i} \left[1+ \frac{B_i-B_j}{1+B_j} \right]\right) \right|\\
& = &  \left|\delta' \left(\frac{A}{1+B_i} \cdot \left[1+\frac{\theta }{1+B_j} \right] \right) \right| \frac{A}{(1+B_i)(1+B_j)}
\end{eqnarray*}
for some $\theta \in (0,1)$,
and therefore
$$
\left| \delta \left(\frac{A}{1+B_i}\right) - \delta \left(\frac{A}{1+B_j}\right) \right| B_i \leq \sup_{t\in [1,2]} \left|\delta' \left(\frac{t A}{1+B_i}  \right) \right| \frac{A}{(1+B_i)}.
$$
Plugging back into (\ref{temp_00_01}), we have
\begin{eqnarray*}
 |\lambda_i-\lambda_j| & \leq & \delta \left(\frac{1+d_1(S_1,S_2)}{1+d_1(S_1,S_i)}\right) \\
&   &+ \sup_{t\in [1,2]} \left|\delta' \left(\frac{t (1+d_1(S_1,S_2))}{1+d_1(S_1,S_i)}  \right) \right| \frac{1+d_1(S_1,S_2)}{(1+d_1(S_1,S_i))}
\end{eqnarray*}
 As $\supp \delta' \subset [1/2,1]$, the second term will vanish when $\mu_i = 1+d_1(S_1,S_i)$ is sufficiently large,
 so we can say that there exists a modulus of continuity $\delta_2$ (see Remark~\ref{rmk-moc})
 such that $|\lambda_i-\lambda_j|\leq \delta_2(\frac{1}{\mu_i})$
and
$$ \|\delta_2\|\leq \|\delta\|_\infty + 2 \|\delta'\|_\infty.
$$
Thus by Proposition~\ref{main_lemma_1}, $\phi_{S_1,S_2}\in \VMO(\Omega)$ with
$$
\|\phi_{S_1,S_2}\|_\BMO \leq C(\|\delta_1\|_\infty + \|\delta_2\|_\infty),
$$
and the right hand side is independent of $S_1$ and $S_2$.
\end{proof}

Now we can prove the part (ii) of the main theorem.

Recall that $\psimuSi = 1$ on $J({S}_i)$, the cube concentric with $S_i$ and of sidelength $\ell(S_i)/4$. In particular, for the function $\phi_{S_1,S_2}$ defined above we have
$$
|(\phi_{S_1,S_2})_{J({S}_1)} - (\phi_{S_1,S_2})_{J({S}_2)}| = |\lambda_1 - \lambda_2| = d_1(S_1,S_2)
$$
On the other hand by Lemma \ref{lem_for_th1} and Lemma 2.1 in \cite{Jones},
\begin{eqnarray*}
|(\phi_{S_1,S_2})_{J({S}_1)} - (\phi_{S_1,S_2})_{J({S}_2)}| & = & |(\Lambda \phi_{S_1,S_2})_{J({S}_1)} - (\Lambda \phi_{S_1,S_2})_{J({S}_2)}| \\
& \leq &
 C \|\Lambda \phi_{S_1,S_2} \|_{\BMO(\Rn)} d_2(J({S}_1),J({S}_2)) \\
 & \leq &C' d_2(S_1,S_2).
\end{eqnarray*}
Hence for some $C'>0$ and any $S_1,S_2$,
$$d_1(S_1,S_2) \leq C' d_2(S_1,S_2).$$

\section{Proof of Proposition~\ref{main_lemma_1}}
The goal of this section is to prove Proposition~\ref{main_lemma_1}.
 In view of Lemma \ref{lem_gluing_dyadic}, we need to show that conditions (\ref{cond_0})-(\ref{cond_3}), imply (\ref{cond_gluing_1}) and (\ref{cond_gluing_2}).
First we note the following general fact.

\begin{lem}
 \label{lem_product_of_moduli}
 Let $\eta_1$ and $\eta_2$ be moduli of continuity. If $\mu_i$ satisfies (\ref{cond_0}), then there exists a modulus $\eta_3$ such that
$$
\eta_1\left(\frac{1}{\mu_i}\right) \cdot \eta_2\left( \frac{2^{\mu_i}}{\ell(S_i)}\cdot t \right) \leq \eta_3(t),
$$
with $\|\eta_3\|_\infty \leq \|\eta_1\|_\infty \cdot \|\eta_2\|_\infty$.
 \end{lem}
 \begin{proof}
Let us consider two cases:

\noindent
\textit{Case 1:} If $\frac{2^{\mu_i}}{\ell(S_i)} \leq t^{-1/2}$ then
$$
\eta_1\left(\frac{1}{\mu_i}\right) \cdot \eta_2\left( \frac{2^{\mu_i}}{\ell(S_i)}\cdot t \right) \leq \|\eta_1\|_\infty  \cdot \eta_2(t^{1/2})
$$

\noindent
\textit{Case 2:} If $\frac{2^{\mu_i}}{\ell(S_i)} > t^{-1/2}$, then by (\ref{cond_0})
$$
t^{-1/2} < \frac{2^{\mu_i}}{\ell(S_i)} \leq 2^{\mu_i} \cdot \frac{1}{L} \cdot 2^{C_2/C_1} 2^{\mu_i/C_1} = C_L \cdot 2^{(1+1/C_1)\mu_i}
$$
If we introduce function $g(x)= C_L 2^{(1+1/C_1)x}$, then $g$ is a strictly increasing function, hence invertible, with $g(x) \ra \infty$ as $x \ra \infty$, and the above can be written as
$$
    t^{-1/2} \leq g(\mu_i)
$$
    or
$$
    g^{-1}(t^{-1/2}) \leq \mu_i.
$$
    Therefore
$$
    \eta_1\left(\frac{1}{\mu_i}\right) \eta_2\left( \frac{2^{\mu_i}}{\ell(S_i)}\cdot t \right) \leq \|\eta_2\|_\infty \cdot \eta_1 \left( \frac{1}{g^{-1}(t^{-1/2})}\right)
$$
and all in all,
$$
    \eta_1\left(\frac{1}{\mu_i}\right) \eta_2\left( \frac{2^{\mu_i}}{\ell(S_i)}\cdot t \right) \leq \max\left( \|\eta_1\| \cdot \eta_2(t^{1/2}), \|\eta_2\|_{\infty} \eta_1 \left( \frac{1}{g^{-1}(t^{-1/2})}\right) \right).
$$
Since $g^{-1}(t^{-1/2}) \ra \infty$ as $t \ra 0$, the right-hand-side can be dominated (see Remark~\ref{rmk-moc})
by a modulus of continuity  $\eta_3$ with $\|\eta_3\|_\infty \leq \|\eta_1\|_\infty \cdot \|\eta_2\|_\infty$.
\end{proof}

\begin{lem}
\label{lem_uniform_modulus}
If $\mu_i$, $\lambda_i$ are such that (\ref{cond_2}) holds, then there exists a modulus
$\eta$ with $\|\eta\|_\infty \leq C\|\delta_1\|_\infty$ such that for all $t$,
$$
\omega_{S_i}(\lambda_i \psi^{\mu_i}_{S_i},t) \leq \eta(t).
$$
\end{lem}

\begin{proof}
By Lemma \ref{lem_bump_modulus} and (\ref{cond_2}),  we have
$$
\omega_{S_i}(\lambda_i \psi^{\mu_i}_{S_i},t)\leq \frac{|\lambda_i|}{\mu_i} \cdot \theta\left(\frac{2^{\mu_i} }{\ell(S_i)} t\right) \leq \delta_1\left(\frac{1}{\mu_i}\right) \cdot \theta\left(\frac{2^{\mu_i} }{\ell(S_i)} t\right).
$$
Then by Lemma \ref{lem_product_of_moduli}, there is a modulus of continuity
$\eta$  with
$\|\eta\|_\infty \leq \|\delta_1\|_\infty\|\theta\|_\infty \leq C\|\delta_1\|_\infty$  such that for all $t$,
$$
\omega_{S_i}(\lambda_i \psi^{\mu_i}_{S_i},t)\leq \eta(t).
 $$
 \end{proof}

Now that we have shown  (\ref{cond_gluing_1}), in order to apply Lemma~\ref{lem_gluing_dyadic} to prove the Proposition
it remains to show that (\ref{cond_gluing_2}) is implied by (\ref{cond_0})-(\ref{cond_3}).   We first establish some preliminary results.

 \begin{lem}
 \label{lem_special_of_claim}
Write $x \in \Rn$ as $(x_1,x_2)$ where $x_2$ denotes a vector in $\R^{n-1}$.
Suppose $S_1,S_2$ are adjacent cubes such that
$$
 S_1\subset \{(x_1,x_2): x_1\leq 0, |x_2|\leq \ell(S_1)/2\},
 $$
$$
S_2\subset \{(x_1,x_2): x_1\geq 0, |x_2|\leq \ell(S_2)/2\}
$$
 and consider adjacent subcubes $Q_1, Q_2$, $Q_i \subset S_i$,  satisfying
 $$
Q_1\subset\{(x_1,x_2):\dist(x,\partial S_1) = |x_1|, \  x_1\geq -\ell(S_1)/4\},
$$
$$
    Q_2\subset\{(x_1,x_2): \dist(x,\partial S_2) = x_1,\  x_1\leq \ell(S_2)/4\},
 $$
    and
$$
 \frac{\ell(Q_i)}{\ell(S_i)}=\frac{\ell(Q_j)}{\ell(S_j)}.
$$
 If $\psi_{S_i} = \psi^{\mu_i}_{S_i}$ for $i=1,2$, where $\mu_i$ satisfy (\ref{cond_1}),
then for some $C>0$
  \begin{equation}
\label{temp_2018}
 |(\psi_{S_1})_{Q_1} - ( \psi_{S_2})_{Q_2} | \leq \frac{C}{\max \mu_i}.
\end{equation}
\end{lem}

 \begin{proof}
Since on $Q_i$,  $i=1,2$, $d(x,\partial S_i) = |x_1|$,
the functions $\psi_{S_i}(x)$ are functions of the variable $x_1$ only and
 $$
\psi_{S_i}(x) = \left(1+ \frac{1}{\mu_i} \log\left(\frac{4|x_1|}{\ell(S_i)}\right) \right)_+ , \quad x \in Q_i.
$$
Then for $i = 1,2$
\begin{eqnarray*}
(\psi_{S_i})_{Q_i} & =  & \fint_{Q_i} \left(1 + \frac{1}{\mu_i} \log\left(\frac{4|x_1|}{\ell(S_i)}\right)\right)_+ dx_1 dx_2\\
& =  & \frac{1}{\ell(Q_i)} \int_0^{\ell(Q_i)} \left(1 + \frac{1}{\mu_i} \log\left(\frac{4|x_1|}{\ell(S_i)}\right)\right)_+ dx_1\\
& =  & \frac{\ell(S_i)}{4 \ell(Q_i)} \int_0^{\frac{4 \ell(Q_i)}{\ell(S_i)}} \left(1+ \frac{1}{\mu_i} \log t \right)_+ dt.
\end{eqnarray*}
Let $l = \frac{4 \ell(Q_i)}{\ell(S_i)}$, so we can write
$$
|(\psi_{S_1})_{Q_1}-(\psi_{S_2})_{Q_2}| = \frac{1}{l} \int_0^{l} \left(1+ \frac{1}{\mu_1} \log t  \right)_+  - \left(1+ \frac{1}{\mu_2} \log t  \right)_+ dt =: I.
 $$
Without loss of generality we can assume that $\mu_2\leq \mu_1$. Then by condition (\ref{cond_1})
$$
 \mu_2\leq \mu_1\leq \mu_2+C
$$
and $2^{-\mu_1}$ and $2^{-\mu_2}$ are comparable.
If $l<2^{-\mu_1}$, then $I=0$ and there is nothing to prove. Otherwise consider two cases:

\noindent
\textit{Case 1:} If $2^{-\mu_1} \leq l\leq 2^{-\mu_2}$, then $l$ is comparable to $2^{-\mu_1}$ and $2^{-\mu_2}$ so
\begin{eqnarray*}
 |I| =\frac{1}{l} \int_0^{l} \left(1+ \frac{1}{\mu_1} \log t  \right)_+  dt & \leq&  \frac{1}{l} \int_{2^{-\mu_1}}^{2^{-\mu_2}} \left(1+ \frac{1}{\mu_1} \log t  \right)  dt \\
 & \leq&     \frac{C}{\mu_1} (\mu_1-\mu_2) \leq \frac{C'}{\mu_1}.
\end{eqnarray*}

\noindent
\textit{Case 2:} If $l> 2^{-\mu_2}$, then
  $$
    I=\frac{1}{l} \int_{0}^{2^{-\mu_2}} \left(1+ \frac{1}{\mu_1} \log t  \right)_+  dt +  \left( \frac{1}{\mu_1} - \frac{1}{\mu_2} \right) \frac{1}{l} \int_{2^{-\mu_2}}^{l}  \log t \ dt.
 $$
    As in the previous case, the absolute value of the first term is bounded by $\frac{C}{\mu_1}$. For the second term, we have
$$
    \left|\left( \frac{1}{\mu_1} - \frac{1}{\mu_2} \right) \frac{1}{l} \int_{2^{-\mu_2}}^{l}  \log t \ dt \right| =  \frac{\mu_1-\mu_2}{\mu_1 \mu_2}  \left|\log l -1 + \frac{1+\mu_2}{l 2^{\mu_2}} \right| \leq \frac{C }{\mu_1}.$$
\end{proof}

\begin{lem}
\label{lem_key_one}
Let conditions (\ref{cond_0})-(\ref{cond_3}) hold. Then there exists a modulus of continuity
$\eta$
such that for all adjacent Whitney cubes $S_i$, $S_j$ and $Q_i$ and $Q_j$ touching
subcubes with $Q_i\subset S_i$, $Q_j\subset S_j$,
 $\ell(Q_i),\ell(Q_j)\leq t$ and
\begin{equation}
\label{eqn-proportional}
\frac{\ell(Q_i)}{\ell(S_i)}=\frac{\ell(Q_j)}{\ell(S_j)},
\end{equation}
we have
\begin{equation}
\label{temp_to_show_1}
|\lambda_i(\psi^{\mu_i}_{S_i})_{Q_i} - \lambda_j( \psi^{\mu_j}_{S_j})_{Q_j} | \leq \eta(t).
\end{equation}
Moreover,
\begin{equation}
\label{eqn-Linfty}
\|\eta\|_\infty \leq  C(\|\delta_1\|_\infty + \|\delta_2\|_\infty)
\end{equation}
with $C$
depending on the constants in \eqref{cond_0} and \eqref{cond_1}.
\end{lem}

\begin{proof}  Let $S_i,S_j, S_i$ and $Q_j$ be as in the hypotheses.
Without loss of generality we may assume that $\mu_i\geq \mu_j$.
By the triangle inequality,
$$
|\lambda_i(\psi^{\mu_i}_{S_i})_{Q_i} - \lambda_j( \psi^{\mu_j}_{S_j})_{Q_j} |
\leq |\lambda_i-\lambda_j| |(\psi^{\mu_j}_{S_j})_{Q_j}|
+ |\lambda_i| |(\psi^{\mu_i}_{S_i})_{Q_i}-(\psi^{\mu_j}_{S_j})_{Q_j}|=: A + B.
$$
By condition (\ref{cond_3}) and our assumption, we can write
$$
A \leq |(\psi^{\mu_j}_{S_j})_{Q_j}| \delta_2\left(\frac{1}{\mu_i} + \frac{1}{\mu_j}\right)
\leq  |(\psi^{\mu_j}_{S_j})_{Q_j}|\delta_2\left(\frac{2}{\mu_j}\right)
$$
By the size and support properties of the bump function,
$|(\psi^{\mu_j}_{S_j})_{Q_j}| \leq 1$ and,
since $Q_j$ lies along one of the faces of $S_j$,
$|(\psi^{\mu_j}_{S_j})_{Q_j}| = 0$  if  $\ell(Q_j)\leq 2^{-\mu_j-2} \ell(S_j)$.
Letting $\eta_1(t) =  \delta_2(2t)$ and $\eta_2$ be
any modulus of continuity satisfying $\eta_2(s)=0$ for $s\leq 1/8$ and $\eta_2(s)=1$
for $s \geq 1/4$, we may write
$$A \leq \eta_1\left(\frac{1}{\mu_j}\right)
\eta_2\left(\frac{2^{\mu_j} \ell(Q_j)}{\ell(S_j)}\right).
$$
Hence by Lemma \ref{lem_product_of_moduli}, there exists $\eta_3$
such that $\|\eta_3\|_\infty \leq \|\delta_2\|_\infty$ and
\begin{equation}
\label{temp_212}
A\leq \eta_3(\ell(Q_j)) \leq  \eta_3(t).
\end{equation}

To deal with second term, apply condition (\ref{cond_2}) to get
$$
B = \frac{|\lambda_i|}{\mu_i}\; (\mu_i |(\psi^{\mu_i}_{S_i})_{Q_i}-(\psi^{\mu_j}_{S_j})_{Q_j}|)
 \leq  \delta_1\left( \frac{1}{\mu_i}\right)\;
\mu_i |(\psi^{\mu_i}_{S_i})_{Q_i}-(\psi^{\mu_j}_{S_j})_{Q_j}|.
$$
As above, since by our assumption $2^{-\mu_i} \leq 2^{-\mu_j}$, if
$$
\frac{\ell(Q_i)}{\ell(S_i)}=\frac{\ell(Q_j)}{\ell(S_j)} \leq 2^{-\mu_i-2},
$$
then the support properties of the bump functions give
$(\psi^{\mu_i}_{S_i})_{Q_i} = 0 = (\psi^{\mu_j}_{S_j})_{Q_j}$
and
$$\mu_i |(\psi^{\mu_i}_{S_i})_{Q_i}-(\psi^{\mu_j}_{S_j})_{Q_j}|.$$
Assuming this quantity is also
bounded, that is,  for some $C_0>0$
\begin{equation}
\label{temp_claim}
\mu_i |(\psi^{\mu_i}_{S_i})_{Q_i}-(\psi^{\mu_j}_{S_j})_{Q_j}|\leq C_0,
\end{equation}
we can apply the same reasoning as above to get a modulus of continuity $\tilde{\eta_2}$
with $\|\tilde{\eta_2}\|_\infty \leq C_0$ such that
$$
B \leq C_0 \delta_1\left( \frac{1}{\mu_i}\right) \; \eta\left( \frac{2^{\mu_i} \ell(Q_i)}{\ell(S_i)} \right).
$$
Then Lemma \ref{lem_product_of_moduli} gives a modulus $\tilde{\eta_3}$  with
$\|\tilde{\eta_3}\|_\infty \leq C_0 \|\delta_1\|_\infty$ and
$$B\leq \tilde{\eta_3}(\ell(Q_i))\leq \tilde{\eta_3}(t).
$$
Combining this with (\ref{temp_212}) and letting $\eta = \max(\eta_3, \tilde{\eta_3})$,
we can complete the proof, assuming that (\ref{temp_claim}) holds.

Note that estimate (\ref{temp_claim})
is established in the special setting of Lemma \ref{lem_special_of_claim}.
Thus all that remains is to show that we can always reduce to that setting.
We may always assume that $S_i\subset \{(x_1,x_2): x_1\leq 0, |x_2|\leq \ell(S_i)/2\}$,
$S_j\subset \{(x_1,x_2): x_1\geq 0, |x_2|\leq \ell(S_j)/2\}$. What needs to be shown
is that it may also be assumed that
$$
Q_i\subset\{(x_1,x_2): \dist(x,\partial S_i) = |x_1|, \  x_1\geq -\ell(S_i)/4\}
$$
and
$$
Q_j\subset\{(x_1,x_2): \dist(x,\partial S_j) = |x_1|,\  x_1\leq \ell(S_j)/4\}.
$$

By symmetry, the frustum $\{x \in S_i: \dist(x,\partial S_i) = |x_1|, x_1\geq -\ell(S_i)/4\}$
has $1/2n$ of the volume of $S_i\setminus J(S_i)$ and lies along the face $x_1 = 0$ with
angles bounded below (depending on the dimension),
so if $Q_i$ does not satisfy the inclusion above, some portion bounded below of it does, and
similarly for $Q_j$.  That is, there is $c_n\in (0,1)$ which depends on $n$ only such
that if at least one of the inclusions does not hold, then we can find cubes $Q_i'\subset Q_i$
and $Q_j'\subset Q_j$ with
$$\frac{\ell(Q'_i)}{\ell(Q_i)}=\frac{\ell(Q'_j)}{\ell(Q_j)}= c_n$$
for which the above conditions
are satisfied.  Write
$$
|(\psi^{\mu_i}_{S_i})_{Q_i} - ( \psi^{\mu_j}_{S_j})_{Q_j} |   \leq
 |(\psi^{\mu_i}_{S_i})_{Q_i}-(\psi_{S_i})^{\mu_i}_{Q'_i}| + |(\psi_{S_j})^{\mu_j}_{Q_j'}
-(\psi^{\mu_j}_{S_j})_{Q_j}|
+  |(\psi^{\mu_i}_{S_i})_{Q_i'}-(\psi^{\mu_j}_{S_j})_{Q_j'}|.
$$
Then by Lemma \ref{lem_d2_proper} and Lemma \ref{lem_bump_modulus}
$$
|(\psi^{\mu_i}_{S_i})_{Q_i}-(\psi^{\mu_i}_{S_i})_{Q'_i}|
\leq C \omega_{S_i}(\psi^{\mu_i}_{S_i},\ell(Q_i)) \leq \frac{C}{\mu_i}
$$
and similarly
$$
|(\psi_{S_j})^{\mu_j}_{Q_j'} -(\psi^{\mu_j}_{S_j})_{Q_j}|
\leq C \omega_{S_j}(\psi^{\mu_j}_{S_j}, \ell(Q_j)) \leq \frac{C}{\mu_j} \leq \frac{C}{\mu_i}
$$
by our assumption. Applying
Lemma \ref{lem_special_of_claim} to bound
$$|(\psi^{\mu_i}_{S_i})_{Q_i'}-(\psi^{\mu_j}_{S_j})_{Q_j'}| \leq
\frac{C}{\max (\mu_i, \mu_j)} = \frac{C}{\mu_i},$$
we have that (\ref{temp_claim}) holds.
\end{proof}

The proof of the Proposition via Lemma~\ref{lem_gluing_dyadic} will be complete once we verify the following.

\begin{lem}
If (\ref{cond_0})-(\ref{cond_3}) hold, then \eqref{cond_gluing_2} is satisfied for $\phi_i = \lambda_i (\psi^{\mu_i}_{S_i})_{Q_i}$,
with a modulus of continuity $\eta$ satisfying \eqref{eqn-Linfty}.
\end{lem}

\begin{proof}
We need to show that there exists a modulus of continuity $\eta$ such that for
any two adjacent Whitney cubes $S_i,S_j$ and touching dyadic cubes $Q_i \subset S_i$, $Q_j \subset S_j$ with $\ell(Q_i)=\ell(Q_j)\leq t$,
$$
|\lambda_i (\psi^{\mu_i}_{S_i})_{Q_i} - \lambda_j (\psi^{\mu_j}_{S_j})_{Q_j} | \leq \eta(t).
$$
The difference with Lemma~\ref{lem_key_one} is that we are not assuming \eqref{eqn-proportional}, so $\ell(S_i) \neq \ell(S_j)$.
Without loss of generality we may assume that $\ell(S_i) < \ell(S_j)$, hence $\frac{\ell(Q_i)}{\ell(S_i)} > \frac{\ell(Q_j)}{\ell(S_j)}$.
Recall that since $S_i$ and $S_j$ are adjacent Whitney cubes, we have $\ell(S_j) \leq 4\ell(S_i)$.

Let $Q_j'$  be a dyadic subcube of $S_j$ containing $Q_j$, hence touching $Q_i$, and of sidelength
$$
\ell(Q_j')= \frac{\ell(Q_i) \ell(S_j)}{\ell(S_i)}.
$$
Note that this means $\ell(Q'_j) \leq 4 \ell(Q_i) = 4\ell(Q_j)$.

By the triangle inequality,
$$
|\lambda_i (\psi^{\mu_i}_{S_i})_{Q_i} - \lambda_j (\psi^{\mu_j}_{S_j})_{Q_j} | \leq |\lambda_i (\psi^{\mu_i}_{S_i})_{Q_i} - \lambda_j (\psi^{\mu_j}_{S_j})_{Q'_j} | + |\lambda_j (\psi^{\mu_j}_{S_j})_{Q'_j} - \lambda_j (\psi^{\mu_j}_{S_j})_{Q_j} |,
$$
By Lemma \ref{lem_key_one}, for some modulus $\eta_1$ satisfying  \eqref{eqn-Linfty} we have
$$
 |\lambda_i (\psi^{\mu_i}_{S_i})_{Q_i} - \lambda_j (\psi^{\mu_j}_{S_j})_{Q'_j} | \leq \eta_1(t) .
$$
By Lemmas~\ref{lem_d2_proper} and \ref{lem_uniform_modulus} there is modulus of continuity $\eta_2$ with $\|\eta_2\|_\infty \leq C\|\delta_1\|_\infty$ such that
$$
 |\lambda_j (\psi^{\mu_j}_{S_j})_{Q'_j} - \lambda_j (\psi^{\mu_j}_{S_j})_{Q_j} | \leq C\eta_2(\ell(Q'_j)) \leq C\eta_2(4t).
 $$
 Letting $\eta(t) = \max(\eta_1(t), \eta_2(4t))$, we have that  \eqref{cond_gluing_2} and \eqref{eqn-Linfty}
 hold.
 \end{proof}

\section{Proof of Proposition~\ref{main_lemma_2}}
We first establish the Proposition under stronger assumptions.

\begin{lem}
\label{lem_log_improvment_easy_version}
Let $S,S_L\in E$ and $\ell(S_L)=L$, where $L$ is as in \eqref{eqn-L}.
Let $\phi\in \VMO(\Omega)$ and $\omegabar$ be the least concave majorant of $\omega_{\Omega}(\phi, \cdot)$.
Given a Whitney chain
$
S=S_0, S_1, \dots, S_m=S_L
$
satisfying
\begin{equation}
\label{eqn-doubling}
\ell(S_j)\geq 2 \ell(S_{j-1}), \quad j=1,2\dots, m,
\end{equation}
we have
$$
|\phi_S - \phi_{S_L}| \leq C \log\frac{L}{\ell(S)}  \cdot \omegabar\left(\frac{4 L }{\log\frac{L }{\ell(S)}} \right).
$$
\end{lem}

\begin{proof}
By \eqref{eqn-doubling} and property \eqref{eqn-adjacent} of adjacent Whitney cubes,
$$
2 \leq \frac{\ell(S_{j})}{\ell(S_{j-1})} \leq 4, \quad j=1,2\dots, m,
$$
and therefore
$$\frac{L}{\ell(S)} =  \frac{\ell(S_{m})}{\ell(S_{0})} = \prod_{j = 1}^m  \frac{\ell(S_{j})}{\ell(S_{j-1})} \in [2^m, 4^m],$$
or taking logs (to the base $2$),
\begin{equation}
\label{temp_2122}
\frac{1}{2} \log \frac{L}{\ell(S)} \leq m \leq \log \frac{L}{\ell(S)}.
\end{equation}
By Lemma \ref{lem_d1},
$$
|\phi_{S_j}-\phi_{S_{j-1}}| \leq C \omega_{\Omega}(\phi,\ell(S_j)) \leq C \omega_{\Omega}(\phi,2^{j-m} L).
$$
Hence, by the triangle inequality,
$$
|\phi_S - \phi_{S_L}| \leq C \sum_{j=1}^m \omega_{\Omega}(\phi,2^{j-m} L).
$$
If $\omegabar$ is a concave majorant of $\omega_{\Omega}(\phi, \cdot)$, then
$$
|\phi_S - \phi_{S_L}| \leq C m \; \left [\frac{1}{m} \sum_{j=1}^m \omegabar(2^{j-m} L) \right] \leq Cm \; \omegabar \left ( \sum_{j=1}^m  \frac{2^{j-m}}{m}  L \right) \leq Cm \; \omegabar \left (\frac{2 L}{m}   \right).
$$
Using (\ref{temp_2122}) we complete the proof.
\end{proof}

In order to generalize this lemma we need the following result of Jones (see Lemma 2.6 in \cite{Jones}).
\begin{lem}
\label{lem_Jones_about_chains}
Let $\Omega$ be a domain satisfying the Jones condition (\ref{eqn-JonesCond}) with constant $\kappa$, $S_0,S_j \in E$ and $S_0,S_1,\dots, S_j$ be a shortest Whitney chain connecting $S_0$ and $S_j$.  If $\ell(S_0)<\ell(S_j)$, then there exists an integer $\alpha\leq \min (j, \kappa^2)$ such that
\begin{equation}
\label{eqn-proportional2}
2\leq \frac{\ell(S_\alpha)} {\ell(S_0)}\leq 4.
\end{equation}
\end{lem}

\begin{proof}[Proof of Proposition \ref{main_lemma_2}]
Let $S=S_0, S_1, \dots , S_m=S_L$ be a shortest Whitney chain connecting $S$ and $S_L$.
We claim that there are integers
$$
0=\alpha_0<\alpha_1<\dots < \alpha_M \leq m
$$
such that
\begin{equation}
\label{eqn-equal-lengths}
\ell(S_{\alpha_M})=L,
\end{equation}
\begin{equation}
\label{temp_alpha_cond_1}
d_1(S_{\alpha_{i}},S_{\alpha_{i-1}})\leq \kappa^2,
\end{equation}
\begin{equation}
\label{temp_alpha_cond_2}
2 \leq \frac{\ell(S_{\alpha_{i}})}{\ell(S_{\alpha_{i-1}})} \leq 4.
\end{equation}

The existence of $\alpha_j$ follows from Lemma \ref{lem_Jones_about_chains}. If $\ell(S_0)=\ell(S_L)=L$, we simply put $M=0$. Otherwise $\ell(S_0) < \ell(S_L)$ so by Lemma~\ref{lem_Jones_about_chains}, applied to $S_0 = S$ and $S_j = S_L$, there exists an integer $\alpha \leq \kappa^2$ such that \eqref{eqn-proportional2} holds.
We set $\alpha_1$ to be the smallest such $\alpha$.
If $\ell(S_{\alpha_1})=L$, then we put $M=1$ and stop.  Otherwise we can again evoke Lemma \ref{lem_Jones_about_chains}, this time with $S_0 = S_{\alpha_1}$,
to obtain
$\alpha_2$ with $\alpha_2 - \alpha_1 \leq \kappa^2$.  Continuing in this way, we have $S_{\alpha_i}$, $i = 1,\ldots, M$, where $M$ is the first $i$ with $\ell(S_{\alpha_i})=L$.
By construction, \eqref{eqn-equal-lengths} and \eqref{temp_alpha_cond_2} hold.

Moreover, note that along a shortest Whitney chain, the distance $d_1$ is additive, that is, the part of the given Whitney chain between $S_{\alpha_{i-1}}$ and $S_{\alpha_{i}}$ also forms a shortest chain between these two cubes, so $d_1(S_{\alpha_{i}},S_{\alpha_{i-1}}) = \alpha_i -\alpha_{i-1}$ and \eqref{temp_alpha_cond_1} holds.
Finally, by the minimality of the choice of $\alpha_i$ at each stage, and property \eqref{eqn-adjacent} of adjacent Whitney cubes, we must have that
$$
\frac{\ell(S_j)}{\ell(S_{\alpha_{i-1}})} \leq 1, \quad \alpha_{i-1} < j < \alpha_i.
$$
Thus $S_{\alpha_{i}}$ is the largest cube in a shortest Whitney chain between $S_{\alpha_{i-1}}$ and $S_{\alpha_{i}}$.

Now that we have the $\alpha_i$, the proof is almost the same as the one of Lemma \ref{lem_log_improvment_easy_version}. Indeed, in this case \eqref{temp_alpha_cond_2} implies
\begin{equation}
 \label{temp_2123}
\frac{1}{2} \log\frac{L}{\ell(S)} \leq M \leq \log\frac{L}{\ell(S)}
\end{equation}
and we can write, setting $S_{\alpha_{M+1}} = S_L$ and applying Lemma 2, \eqref{temp_alpha_cond_1}, \eqref{temp_alpha_cond_2}  and the Jones condition,
\begin{eqnarray*}
|\phi_{S}-\phi_{S_L}| & \leq & \sum_{i=1}^{M+1}|\phi_{S_{\alpha_i}}-\phi_{S_{\alpha_{i-1}}}| \\
& \leq &  \sum_{i=1}^{M+1}  d_1(S_{\alpha_{i}},S_{\alpha_{i-1}}) \omega_{\Omega}(\phi,\ell(S_{\alpha_i})) \\
& \leq &  \sum_{i=1}^{M}  \kappa^2 \omega_{\Omega}(\phi,2^{i-M}L)  + \kappa d_2(S_{\alpha_M},S_L) \omega_{\Omega}(\phi,L)
\end{eqnarray*}
Recalling the definition of $d_2$, we have
\begin{eqnarray*}
d_2(S_{\alpha_M},S_L)  & = &\left| \log \frac{\ell(S_{\alpha_M})}{\ell(S_L)} \right| + \log \left( 2 + \frac{d(S_{\alpha_M},S_L)}{\ell(S_{\alpha_M}) + \ell(S_L)}\right)\\
& \leq & 0 + \log \left( 2 + \frac{\diam(\Omega)}{2L}\right) = C_\Omega.
\end{eqnarray*}
Now, as in the proof of Lemma~\ref{lem_log_improvment_easy_version}, we have
\begin{eqnarray*}
|\phi_{S}-\phi_{S_L}| & \leq & C_\kappa \sum_{i=1}^{M+1} \omega_{\Omega}(\phi,\min(2^{i-M},1)L)\\
& \leq & C_\kappa (M+1) \omegabar\left(\frac{3L}{M+1}\right),
\end{eqnarray*}
where $\omegabar$ is a convex majorant of $\omega_\Omega(\phi,\cdot)$.
By (\ref{temp_2123}), this completes the proof of the proposition.
\end{proof}

\section{Proof of Proposition~\ref{main_lemma_3}}
Assume the hypotheses of Proposition~\ref{main_lemma_3}.  All the lemmas in this section
refer to the notation in the statement of the Proposition.

We proceed following the ideas of the proof of Lemma 2.11 in \cite{Jones}.
We will need the following result (Lemma 2.10 in \cite{Jones}).
 \begin{lem}
\label{lem_dist}
Let $S'\in E'$ and $S\in E$ be matching cubes. Then
$$
d(S,S') \leq 65 \kappa^2 \ell(S').
$$
 \end{lem}

The following lemma shows that  \eqref{eqn-lemma3hyp1}
is independent of the choice of matching cube $S$ of $S'$.

\begin{lem}
\label{Matching cubes}
If $S_1, S_2$ are matching cubes to $S'$, then for some $C_1,C_2 > 0$, which depend only on $\kappa$,
$$
|(\phi_1)_{S_1} - (\phi_1)_{S_2}| \leq C_1 \omega_{\Omega}(\phi_1,C_2  \ell(S_1)).
$$
\end{lem}
\begin{proof}
Recalling (\ref{matching}), for $i = 1, 2$ we have $\ell(S')\leq \ell(S_i)\leq 2 \ell(S')$ and therefore
\begin{equation}
\label{here_2}
\frac{1}{2}\leq \frac{\ell(S_1)}{\ell(S_2)} \leq 2.
\end{equation}
By Lemma~\ref{lem_dist} and  \eqref{matching}, for both cubes $S_i$ we have
$$
d(S',S_i) \leq 65 \kappa^2 \ell(S') \leq 130 \kappa^2 \ell(S_i).
$$
Thus by the triangle inequality,
$$
d(S_1,S_2)  \leq 130 \kappa^2 (\ell(S_1)+\ell(S_2))
$$
or equivalently
\begin{equation}
\label{here_3}
\frac{d(S_1,S_2)}{\ell(S_1)+\ell(S_2)} \leq 130 \kappa^2.
\end{equation}
Combining \eqref{here_2} and \eqref{here_3} and recalling that $\Omega$ satisfies the Jones condition, we obtain
$$
d_1(S_1,S_2)\leq \kappa d_2(S_1,S_2) \leq C_\kappa
$$
for some $C_\kappa>0$. The proof is then completed by Lemma \ref{lem_d1}.
\end{proof}

\begin{lem}
\label{lem_used_twice2}
Given $C_0>0$, there is a $C>0$ and a modulus of continuity $\eta'$ such that
\begin{equation}
\label{eqn-Linfty'}
\|\eta'\|_{\infty}\leq C (\|\omega_{\Omega}(\phi_1,\cdot) \|_\infty + \|\omega_{\Omega'}(\phi_2,\cdot) \|_\infty + \|\eta \|_\infty)
\end{equation} and
$$
 |\Phi_{Q_1} - \Phi_{Q_2}| \leq \eta'(\ell(Q_1))
$$
uniformly for $Q_i\subset S_i$, $i =1,2$, with $S_i \in E\cup E'$,
\begin{equation}
\label{temp_05}
1\leq \frac{\ell(S_i)}{\ell(Q_i)}\leq C_0,
\end{equation}
and
\begin{equation}
\label{temp_06}
d_2(Q_1,Q_2) \leq C_0.
\end{equation}
 \end{lem}

 \begin{proof}
Note that conditions \eqref{temp_05} and \eqref{temp_06} imply that the four sidelengths $\ell(Q_i)$, $\ell(S_i)$, $i = 1,2$, are all comparable.
We consider three possibilities

\noindent
\textit{Case 1:} $Q_1,Q_2 \subset \Omega$.
 In this case by Lemma \ref{lem_d1}, the Jones condition and (\ref{temp_06}),
\begin{eqnarray*}
 |\Phi_{Q_1}-\Phi_{Q_2}|  =  |(\phi_1)_{Q_1}-(\phi_1)_{Q_2}| & \leq & C d_1(Q_1,Q_2) \omega_\Omega(\phi_1,4^{d_1(Q_1,Q_2)} \ell(Q_1)) \\
 & \leq &
  C'' \omega_\Omega(\phi_1,C \ell(Q_1)).
\end{eqnarray*}

\noindent
\textit{Case 2:} $Q_1\subset S_1 \in E$ and $Q_2\subset S'_2 \in E'$.
Then
\begin{equation}
\label{temp_07}
d(S_1,S_2') \leq d(Q_1,Q_2) \leq C' \ell(Q_1),
\end{equation}
where we have used \eqref{temp_06} and the definition of $d_2$. Let $S_2\in E$ be a matching  cube of $S'_2$.  By  \eqref{matching},
$\ell(S_2)$ is comparable with the other four sidelengths.
Lemma~\ref{lem_dist}, \eqref{temp_07} and \eqref{temp_05} give
$$
 d(S_1,S_2) \leq d(S_1,S_2') + d(S_2,S_2')  \leq C' \ell(Q_1) + C'_\kappa \ell(S_2')  \leq C''_\kappa (\ell(S_1)+\ell(S_2)).
$$
From this and the fact that $\ell(S_1)$ and $\ell(S_2)$ are comparable, we have
\begin{equation}
\label{temp_08}
d_2(S_1,S_2) \leq C_\kappa.
\end{equation}
 Finally, by the triangle inequality,
$$
 |\Phi_{Q_1}-\Phi_{Q_2}| \leq  \sum_{i=1,2} |(\phi_i)_{Q_i}-(\phi_i)_{S_i}| + |(\phi_1)_{S_1}-(\phi_1)_{S_2}| + |(\phi_1)_{S_2}-(\phi_2)_{S_2'}|.
$$
By Lemma~\ref{lem_d2_proper}  and \eqref{temp_05}, we can bound the first two terms by $\omega_{\Omega}(\phi_1,\ell(S_1))+\omega_{\Omega'}(\phi_2,\ell(S_2))$. Because of (\ref{temp_08}), the third term can be handled as in Case 1 above, and is bounded by $C \omega_{\Omega}(\phi_1,(C \ell(Q_1)))$. The fourth term is controlled by the hypothesis \eqref{eqn-lemma3hyp1}.

\noindent
\textit{Case 3:} $Q_1\subset S_1',Q_2\subset S_2'$ where $S_i'\in E'$.
Let $S_1$ and $S_2$ be matching cubes of $S_1'$ and $S_2'$, respectively. Then the six cubes $Q_i, S_i$ and $S_i'$, $i = 1,2$, have comparable sidelengths. As in Case 2,
\begin{eqnarray*}
d(S_1,S_2) & \leq & d(S_1,S_1') + d(S_1',Q_2) + d(Q_2,Q_1)+ d(Q_1,S'_2)+ d(S_2',S_2) \\
& \leq & d(S_1,S_1') + d(S_2',S_2) + 3 d(Q_1,Q_2) \leq C'_\kappa \ell(Q_1)
\end{eqnarray*}
 and therefore
 \begin{equation}
 \label{temp_19}
 d_2(S_1,S_2) \leq C_\kappa.
\end{equation}
Finally,
$$
 |\Phi_{Q_1}-\Phi_{Q_2}| \leq \sum_{i=1,2} \left[|(\phi_2)_{Q_i}-(\phi_2)_{S_i'}| + |(\phi_1)_{S_i}-(\phi_2)_{S_i'}|\right]+|(\phi_1)_{S_1}-(\phi_1)_{S_2}|,
$$
where by \eqref{temp_05}
$$
|(\phi_2)_{Q_i}-(\phi_2)_{S_i'}| \leq C \omega_{\Omega'}(\phi_2,\ell(S'_i)).
$$
Moreover by Lemma \ref{lem_d1} and (\ref{temp_19})
$$
|(\phi_1)_{S_1}-(\phi_1)_{S_2}| \leq C \omega_{\Omega}(\phi_1,C \ell(Q))
$$
and by hypothesis \eqref{eqn-lemma3hyp1}
$$
|(\phi_1)_{S_i}-(\phi_2)_{S_i'}| \leq \eta(\ell(S'_i)).
$$
 \end{proof}

\begin{proof}[Proof of Proposition \ref{main_lemma_3}]
By Lemma \ref{lem_main_one}, we need to establish the existence of a modulus of continuity $\delta$ such that
\begin{equation}
\label{temp_01}
\sup_{\substack{\text{dyadic } Q\subset \Omegatil\\ \ell(Q)\leq t}} \fint_Q |\Phi(x) - \Phi_Q| dx \leq \delta(t)
\end{equation}
and
 \begin{equation}
 \label{temp_02}
|\Phi_{Q_1} - \Phi_{Q_2}|\leq \delta(t),
\end{equation}
for any adjacent dyadic cubes $Q_1,Q_2$ with $\ell(Q_1)=\ell(Q_2)\leq t$.

To prove \eqref{temp_01}, it is enough to consider
$$
\fint_Q |\Phi(x) - \Phi_Q| dx
$$
for $Q$ that does not lie in any $S\in E$ nor in any $S'\in E'$, for otherwise $\delta(t)$ can be chosen as $\omega_{\Omega}(\phi_1,t) + \omega_{\Omega'}(\phi_2,t)$.

Let $Q\subset \Omegatil $ be a given dyadic cube such that $Q\not \subset S$ for any $S\in E\cup E'$.
It is shown in the proof of Lemma 2.11 in \cite{Jones} that for any such $Q$, there exists $r\in \mathbb{N}$ depending only on the geometry of $\Omega$ (namely
$r \approx \log \kappa$, where $\kappa$ is the constant in the Jones condition \eqref{eqn-JonesCond}) and a  sequence
$$
\cF = \bigcup_{m = 1}^\infty \cF_m, \quad \cF_m = \{Q^m_j\}
$$
of collections of dyadic subcubes of $Q$ with disjoint interiors
such that
\begin{enumerate}[label=(\roman*)]
\item each $Q^m_j \in \cF_m$ is contained in some $S\in E\cup E'$, and
$$2^{-mr} \ell(Q) =\ell(Q^m_j)\leq \ell(S) < \ell(Q^{m-1}_k) = 2^{-(m-1)r} \ell(Q),$$
where $Q^{m-1}_k$ is the parent of $Q^m_j$ (note that $Q^{m-1}_k \notin \cF_{m-1}$)
\item there exists a constant $q < 1$ with depends only on $r$ and the dimension
(in \cite{Jones} the proof is done for dimension two and $q = 1 - \kappa^{-8}$)
such that
$$\sum_{\cF_m} |Q^m_j| \leq \sum_{Q^{m-1}_k \notin \cF_{m-1}} |Q^{m-1}_k| \leq q^{m-1} |Q|$$
\item $\Big|Q\setminus \displaystyle{\bigcup_{m=1}^\infty \bigcup_{\cF_m}Q^m_j}\Big|=0$
\item for each $Q^m_j \in \cF_m$ there is a cube $Q^{m-1}_{k'} \in \cF_{m-1}$ such that
$$d(Q^m_j,Q^{m-1}_{k'}) \leq 2^{-r(m-2)}\ell(Q) = 4^r\ell(Q^m_j)$$
(this follows from the fact, explained at the end of the proof of Lemma 2.11 in \cite{Jones}, that both $Q^m_j$
and $Q^{m-1}_{k'}$ are contained in the grandparent cube $Q^{m-2}_{i}$).
\end{enumerate}

Therefore, fixing any $Q_0\in \cF_1$, we have
\begin{eqnarray}
\label{eqn-Phi0}
\lefteqn{\fint_Q |\Phi(x) - \Phi_{Q}| dx  \leq 2\fint_Q |\Phi(x) - \Phi_{Q_0}| dx}  \\
\nonumber
& = & \frac{2}{|Q|} \sum_{m=1}^\infty \sum_{\cF_m} \int_{Q^m_j} |\Phi(x) - \Phi_{Q_0}| dx  \\
\nonumber
& \leq & 2 \sum_{m=1}^\infty \sum_{\cF_m}\frac{|Q^m_j|}{|Q|} \left(\fint_{Q^m_j} |\Phi(x) - \Phi_{Q^m_j}| dx+ |\Phi_{Q^m_j} -\Phi_{Q_0}|\right)\\
\nonumber
& \leq &
2  (\omega_{\Omega}(\phi_1,\ell(Q)) + \omega_{\Omega'}(\phi_2,\ell(Q))) \sum_{m=1}^\infty \sum_{\cF_m}  \frac{|Q^m_j|}{|Q|} \\
\nonumber
&&+ 2\sum_{m=1}^\infty m\sum_{\cF_m} \frac{|Q^m_j|}{|Q|}\frac{|\Phi_{Q^m_j} -\Phi_{Q_0}|}{m} \\
\nonumber
& \leq &
2 (\omega_{\Omega}(\phi_1,\ell(Q)) + \omega_{\Omega'}(\phi_2,\ell(Q))) \sum_{m=1}^\infty q^{m-1}\\
\nonumber
&&+ 2  \left(\sup_{m} \sup_{Q^m_j\in\cF_m}\frac{|\Phi_{Q^m_j} -\Phi_{Q_0}|}{m} \right)  \sum_{m=1}^\infty m q^{m-1},
\end{eqnarray}
where we have used (ii).  Since $q < 1$, the sums converge so it remains to estimate the supremum in the second term.

We can use property (iv) $m$ times to write
$$
 |\Phi_{Q^m_j}-\Phi_{Q_0}| \leq \sum_{i = 1}^m |\Phi_{Q^i_{j_i}}-\Phi_{Q^{i-1}_{j_{i-1}}}|
$$
where $Q^m_{j_m} = Q^m_j$, $Q^0_{j_0} = Q_0$ and for $i = 1, \ldots, m$, $Q^{i-1}_{j_{i-1}}$ is a cube in $\cF_{i-1}$
which is within distance $4^r\ell(Q^i_{j_i})$ of $Q^i_{j_i}$.  By properties (i) and (iv) and the definition of $d_2$, we can apply Lemma~\ref{lem_used_twice2}
to each term in the sum, with a constant $C_0$ depending on $r$, to obtain
$$
 |\Phi_{Q^m_j}-\Phi_{Q_0}| \leq m \; \eta'(\ell(Q))
$$
for some modulus of continuity $\eta'$ satisfying \eqref{eqn-Linfty'}.  This completes the proof of \eqref{temp_01}, with the modulus $\delta$ also
bounded by a constant multiple of $\|\omega_{\Omega}(\phi_1,\cdot) \|_\infty + \|\omega_{\Omega'}(\phi_2,\cdot) \|_\infty + \|\eta \|_\infty$.

It remains to prove \eqref{temp_02}.  Given adjacent dyadic cubes $Q_1,Q_2$ with $\ell(Q_1)=\ell(Q_2)\leq t$,
we can decompose each of them into collections of cubes $\cF_m(Q_i)$, as above.  Fix subcubes $Q_0 \in \cF_1(Q_1)$, $\widetilde{Q_0} \in \cF_1(Q_2)$.
Then we have
\begin{eqnarray*}
|\Phi_{Q_1} - \Phi_{Q_2}|
& \leq &  |\Phi_{Q_1} - \Phi_{Q_0}|  +  |\Phi_{Q_0} - \Phi_{\widetilde{Q_0}}| + |\Phi_{\widetilde{Q_0}} - \Phi_{Q_2}|\\
& \leq &  \fint_{Q_1} |\Phi(x) - \Phi_{Q_0}| dx  + \fint_{Q_2} |\Phi(x) - \Phi_{\widetilde{Q_0}}| dx  +  |\Phi_{Q_0} - \Phi_{\widetilde{Q_0}}|\\
& \leq &  2\delta(t) +  |\Phi_{Q_0} - \Phi_{\widetilde{Q_0}}|
\end{eqnarray*}
where for the last inequality we repeat the estimates for \eqref{eqn-Phi0}.
To estimate $ |\Phi_{Q_0} - \Phi_{\widetilde{Q_0}}|$ we again use Lemma~\ref{lem_used_twice2}, since by (i) $Q_0, \widetilde{Q_0}$ are
contained in Whitney cubes $S$, $\widetilde{S}$, respectively, with
$$\ell(S)/\ell(Q),\; \ell(\widetilde{S})/\ell(\widetilde{Q_0}) \in [1, 2^r).$$
Moreover, $d(Q_0,\widetilde{Q_0}) \leq 2\ell(Q_1)$ and since their sidelengths are equal, this means $d_2(Q_0,\widetilde{Q_0})$ is bounded
by a constant.
 \end{proof}

    \section{Concluding remarks}
    In this paper we focused on the extension problem for a bounded domain $\Omega$, while in \cite{Jones}, the extension is proved first for the case where $\Omega$ contains Whitney cubes of arbitrarily large size, and then the proof is adapted to the case where there is a largest size $L$.  Our proof is based on this latter construction, and depends on the constant $L$.
 To modify the proof to apply to an unbounded domain $\Omega$, one first needs to consider the various definition of vanishing mean oscillation discussed in the introduction
 for the case $\Omega = \Rn$ (Sarason's $\VMO$, Coifman-Weiss $\CMO$, or the nonhomogeneous space $\vmo$) and their analogues for an arbitrary unbounded domain $\Omega$.
 This is the subject of a forthcoming paper.

It is worth emphasizing that in part (ii) of Theorem \ref{thm1} we are only assuming a bounded extension map from $\VMO(\Omega)$ to $\BMO(\Rn)$, which makes sense given that the extension map we exhibited in the proof of part (i), while mapping $\VMO(\Omega)$ to $\VMO(\Rn)$, only preserves (up to a constant) the $L^\infty$ norm of the modulus of continuity, but
does not preserve it in any other sense.  This, as well as the dependence on the constant $L$, is evident especially in Proposition~\ref{main_lemma_2}.

By part (ii) of Theorem \ref{thm1}, if $\Omega$ is a domain that does not satisfy the Jones condition \eqref{eqn-JonesCond}, then there is no bounded linear extension of $\VMO(\Omega)$ to $\BMO(\Rn)$. Nevertheless, there exists a weaker local version of the Jones condition \eqref{eqn-JonesCond}, the so called $(\epsilon,\delta)$ condition  (see \cite{Jones2} p.\ 73 or  Definition 9.49 in \cite{Brudnyi}). It seems natural to ask whether the extension problem in such domains has a positive answer, if one considers not $\VMO(\Omega)$ but some suitably defined local version.

	\providecommand{\bysame}{\leavevmode\hbox to3em{\hrulefill}\thinspace}
\providecommand{\MR}{\relax\ifhmode\unskip\space\fi MR }
\providecommand{\MRhref}[2]{%
  \href{http://www.ams.org/mathscinet-getitem?mr=#1}{#2}
}
\providecommand{\href}[2]{#2}


\begin{thebibliography}{10}
\bibitem{italians1}
Paolo Acquistapace, \emph{On {BMO} regularity for linear elliptic systems},
  Ann. Mat. Pura Appl. (4) \textbf{161} (1992), 231--269.

\bibitem{AuscherRuss}
Pascal Auscher and Emmanuel Russ, \emph{Hardy spaces and divergence operators
  on strongly {L}ipschitz domains of {$\Rn$}}, J. Funct. Anal.
  \textbf{201} (2003), no.~1, 148--184.

\bibitem{Toro}
Jonas Azzam, Steve Hofmann, Jos\'e Mar\'\i a Martell, Kaj Nystr\"om, and
  Tatiana Toro, \emph{A new characterization of chord-arc domains}, J. Eur.
  Math. Soc. (JEMS) \textbf{19} (2017), no.~4, 967--981.

\bibitem{BlascoPerez}
Oscar Blasco and M.~Amparo P\'erez, \emph{On functions of integrable mean
  oscillation}, Rev. Mat. Complut. \textbf{18} (2005), no.~2, 465--477.

\bibitem{Bourdaud}
G\'erard Bourdaud, \emph{Remarques sur certains sous-espaces de {${\rm
  BMO}(\Rn)$} et de {${\rm bmo}(\Rn)$}}, Ann. Inst. Fourier
  (Grenoble) \textbf{52} (2002), no.~4, 1187--1218.

\bibitem{BN2}
Ha\"im Brezis and Louis Nirenberg, \emph{Degree theory and {BMO}. {II}.
  {C}ompact manifolds with boundaries}, Selecta Math. (N.S.) \textbf{2} (1996),
  no.~3, 309--368.

\bibitem{Brudnyi}
Alexander Brudnyi and Yuri Brudnyi, \emph{Methods of geometric analysis in
  extension and trace problems. {V}olume 2}, Monographs in Mathematics, vol.
  103, Birkh\"auser/Springer Basel AG, Basel, 2012.

\bibitem{Buckley}
Stephen~M. Buckley, \emph{Inequalities of {J}ohn-{N}irenberg type in doubling
  spaces}, J. Anal. Math. \textbf{79} (1999), 215--240.

\bibitem{Chang}
Der-Chen Chang, \emph{The dual of {H}ardy spaces on a bounded domain in {${\bf
  R}^n$}}, Forum Math. \textbf{6} (1994), no.~1, 65--81.

\bibitem{CoifmanWeiss}
Ronald~R. Coifman and Guido Weiss, \emph{Extensions of {H}ardy spaces and their
  use in analysis}, Bull. Amer. Math. Soc. \textbf{83} (1977), no.~4, 569--645.

\bibitem{Dafni} Galia Dafni, \emph{Nonhomogeneous div-curl lemmas
and local Hardy spaces},  Adv. Differential Equations  \textbf {10}
(2005), 505--526.

\bibitem{DevoreLorentz}
Ronald~A. DeVore and George~G. Lorentz, \emph{Constructive approximation},
  Grundlehren der Mathematischen Wissenschaften [Fundamental Principles of
  Mathematical Sciences], vol. 303, Springer-Verlag, Berlin, 1993.

\bibitem{Gehring_Osgood}
F.~W. Gehring and B.~G. Osgood, \emph{Uniform domains and the quasihyperbolic
  metric}, J. Analyse Math. \textbf{36} (1979), 50--74 (1980).

\bibitem{Goldberg}
David Goldberg, \emph{A local version of real {H}ardy spaces}, Duke Math. J.
  \textbf{46} (1979), no.~1, 27--42.

\bibitem{Holden}
Peter~J. Holden, \emph{Extension theorems for functions of vanishing mean
  oscillation}, Pacific J. Math. \textbf{142} (1990), no.~2, 277--295.

\bibitem{Jerison_Kenig}
David~S. Jerison and Carlos~E. Kenig, \emph{Boundary behavior of harmonic
  functions in nontangentially accessible domains}, Adv. in Math. \textbf{46}
  (1982), no.~1, 80--147.

\bibitem{JN}
F.~John and L.~Nirenberg, \emph{On functions of bounded mean oscillation},
  Comm. Pure Appl. Math. \textbf{14} (1961), 415--426.

\bibitem{Jones}
Peter~W. Jones, \emph{Extension theorems for {BMO}}, Indiana Univ. Math. J.
  \textbf{29} (1980), no.~1, 41--66.

\bibitem{Jones2}
\bysame, \emph{Quasiconformal mappings and extendability of functions in
  {S}obolev spaces}, Acta Math. \textbf{147} (1981), no.~1-2, 71--88.

\bibitem{Martio}
O.~Martio and J.~Sarvas, \emph{Injectivity theorems in plane and space}, Ann.
  Acad. Sci. Fenn. Ser. A I Math. \textbf{4} (1979), no.~2, 383--401.

\bibitem{ReimannRychener}
Hans~Martin Reimann and Thomas Rychener, \emph{Funktionen beschr\"ankter
  mittlerer {O}szillation}, Lecture Notes in Mathematics, Vol. 487,
  Springer-Verlag, Berlin-New York, 1975.

\bibitem{Sarason}
Donald Sarason, \emph{Functions of vanishing mean oscillation}, Trans. Amer.
  Math. Soc. \textbf{207} (1975), 391--405.

\bibitem{Stein}
Elias~M. Stein, \emph{Singular integrals and differentiability properties of
  functions}, Princeton Mathematical Series, No. 30, Princeton University
  Press, Princeton, N.J., 1970.



\bibitem{italians2}
Maria Transirico, Mario Troisi, and Antonio Vitolo, \emph{B{MO} spaces on
  domains of {${\bf R}^n$}}, Ricerche Mat. \textbf{45} (1996), no.~2, 355--378.

\bibitem{Uchiyama}
Akihito Uchiyama, \emph{On the compactness of operators of {H}ankel type},
  T\^ohoku Math. J. (2) \textbf{30} (1978), no.~1, 163--171.




\end{thebibliography}
\end{document}